\documentclass[11pt]{amsart}

\usepackage{color}
\usepackage{xifthen}
\usepackage{amsmath}
\usepackage[table]{xcolor}
\usepackage{ytableau}
\usepackage{tikz}
\usetikzlibrary{calc, cd}
\usepackage[enableskew]{youngtab}
\usepackage{enumerate}
\usepackage{amssymb}
\usepackage{array}
\usepackage[normalem]{ulem}
\usepackage{scalerel}
\usepackage{tikz-cd}
\usepackage[pagebackref]{hyperref}

\calclayout

\newtheorem{thm}{Theorem}[section]
\newtheorem{lemma}[thm]{Lemma}
\newtheorem{prop}[thm]{Proposition}

\newtheorem{cor}[thm]{Corollary}
\newtheorem{fact}[thm]{Fact}

\theoremstyle{definition} 

\newtheorem{sit}[thm]{Situation}
\newtheorem{defn}[thm]{Definition}

\theoremstyle{remark}

\newtheorem{rem}[thm]{Remark}

\newtheorem{eg}[thm]{Example}

\newcommand{\sbu}{{\raisebox{1pt}{\scaleto{\bullet}{2.5pt}}}}

\newcommand{\eps}{\varepsilon}

\newcommand{\NN}{\mathbb{N}}
\newcommand{\PP}{\mathbb{P}}

\newcommand{\ZZ}{\mathbb{Z}}

\newcommand{\cC}{\mathcal{C}}

\newcommand{\cH}{\mathcal{H}}

\newcommand{\cL}{\mathcal{L}}
\newcommand{\cM}{\mathcal{M}}
\newcommand{\cO}{\mathcal{O}}
\newcommand{\cP}{\mathcal{P}}
\newcommand{\cQ}{\mathcal{Q}}
\newcommand{\cU}{\mathcal{U}}
\newcommand{\cV}{\mathcal{V}}

\newcommand{\on}{\mathrm}

\newcommand{\End}{\on{End}}
\newcommand{\Fr}{\on{Fr}}
\newcommand{\Fix}{\on{Fix}}
\newcommand{\Fl}{\on{Fl}}
\newcommand{\GL}{\on{GL}}

\newcommand{\Hom}{\on{Hom}}

\newcommand{\Pic}{\on{Pic}}
\newcommand{\Spec}{\on{Spec}}

\newcommand{\cpb}{\mathcal{P}^{\sbu}}
\newcommand{\cqb}{\mathcal{Q}^{\sbu}}

\usetikzlibrary{patterns}

\newcommand{\be}{{\overline{\eta}}}
\newcommand{\ess}[2]{\fill[pattern=north west lines] (#2-0.5, 0.5-#1) rectangle (#2+0.5, -0.5-#1);}
\newcommand{\fdot}[2]{\draw[circle, fill=black] (#2,- #1) circle (0.1in);}
\newcommand{\opendot}[2]{\draw[circle] (#2, - #1) circle (0.1in);}
\newcommand{\Ess}{\on{Ess}}
\newcommand{\EssR}{\on{EssR}}
\newcommand{\EssC}{\on{EssC}}

\newcommand{\tw}{\on{tw}}
\newcommand{\sch}{\mathfrak{S}}

\AtBeginDocument{%
   \def\MR#1{}
}

\begin{document}
\title[Versality of Brill-Noether flags]{Versality of Brill-Noether flags and degeneracy loci of twice-marked curves}
\author[N. Pflueger]{Nathan Pflueger}
\email{npflueger@amherst.edu}
\address{
    Amherst College\\
    31 Quadrangle Drive, Amherst, MA 01002}

\begin{abstract}
A Brill-Noether degeneracy locus is closure in $\Pic^d(C)$ of the locus of line bundles with a specified rank function $r(a,b) = h^0(C,L(-ap-bq))$. These loci generalize the classical Brill-Noether loci $W^r_d(C)$ as well as Brill-Noether loci with imposed ramification. For general $(C,p,q)$ we determine the dimension, singular locus, and intersection class of Brill-Noether degeneracy loci, generalizing classical results about $W^r_d(C)$. The intersection class has a combinatorial interpretation in terms of the number of reduced words for a permutation associated to the rank function, or alternatively the number of saturated chains in the Bruhat order. The essential tool is a versality theorem for a certain pair of flags on $\Pic^d(C)$, conjectured by Melody Chan and the author.
\end{abstract}

\maketitle

\section{Introduction}

The classical Brill--Noether and Gieseker--Petri theorems have many proofs and interpretations. Our starting point is an intriguing deformation-theoretic point of view on these theorems. Fix a curve $C$ with two marked points $p,q$. For any line bundle $L$, we may fix $N \gg 0$, and interpret $H^0(C,L)$ as the intersection of two linear subspaces $P = H^0(C, L(Np))$ and $Q = H^0(C,L(Nq))$ within $H = H^0(C,L(Np+Nq))$; since $N \gg 0$ the dimensions of $P,Q,H$ do not depend on $L$. Varying $L$, we obtain a deformation of pairs of subspaces within a vector space. The observation is this: the dimension and smoothness theorems for $W^r_d(C)$ follow from the assumption that \emph{for a general curve, this deformation is versal.}

This paper elaborates on this deformation-theoretic point of view, and thereby obtain some novel results. We extend from Grassmannians to flag varieties. Flag varieties have a beautiful system of Schubert subvarieties indexed by permutations, and this theory is mirrored in Brill--Noether theory of twice-marked curves. A versal deformation plays the pivotal role.

\subsection{The versality theorem}

Fix a smooth twice-marked curve $(C,p,q)$ over an algebraically closed field $k$, and an integer $d \geq 2g-1$. Pushing down a Poincar\'e line bundle on $C \times \Pic^d(C)$ gives a rank $n = d+1-g$ vector bundle $\cH_d$ on $\Pic^d(C)$, with fibers $(\cH_d)_{[L]} \cong H^0(C, L)$. For any divisor $D$ on $C$, denote by $\tw_D: \Pic^d(C) \to \Pic^{d+\deg D}(C)$ the ``twist'' isomorphism given pointwise by $[L] \mapsto [L(D)]$, and define for all $0 \leq a,b \leq d-2g+1$,
$
\cP^a_d = \tw_{ap} \cH_{d-a}$, and $
\cQ_d^b = \tw_{bq} \cH_{d-b}.
$
Regarding these as subbundles of $\cH_d$, we obtain two flags $\cpb_d, \cqb_d$ in $\cH_d$, both of coranks $[0, d-2g+1] \cap \ZZ$. Call these the \emph{degree-$d$ Brill-Noether flags of $(C,p,q)$}.
Section \ref{sec:versality} defines the notion of a \emph{versal pair of flags}. Informally, two flags are versal if their relative positions in any fiber deform ``as generally as possible'' around that fiber. The first main result of this paper is

\begin{thm}
\label{thm:versality}
If $(C,p,q)$ is a general twice-marked curve of genus $g$, then the pair of degree-$d$ Brill-Noether flags $\cpb_d, \cqb_d$ of $(C,p,q)$ is versal.
\end{thm}

This was conjectured in \cite[Conjecture 6.3]{cpRR}.
One might hope to generalize Theorem \ref{thm:versality} to three or more marked points, but this is impossible, even for $g=0$; see \cite[Example 3.4]{cpRR}.

\subsection{The coupled Petri map}
Define, for any line bundle $L$ on a smooth twice-marked curve $(C,p,q)$ and any subset $S \subseteq \ZZ \times \ZZ$,
\begin{equation}
\label{eq:coupledTensors}
T^L_{p,q}(S) = \sum_{(a,b) \in S} H^0(C, L(-ap-bq)) \otimes H^0(C, \omega_C \otimes L^\vee(ap+bq)).
\end{equation}
We write $T^L_{p,q}$ as shorthand for $T^L_{p,q}(\ZZ \times \ZZ)$. To define this sum, regard all terms as subspaces of $H^0(C^\ast, L) \otimes H^0(C^\ast, \omega_C \otimes L^\vee)$, where $C^\ast = C \backslash \{p,q\}$. Call the elements of $T^L_{p,q}$ \emph{coupled tensors}. 

\begin{defn}
\label{defn:coupledPetri}
The \emph{(fully) coupled Petri map} of $L$ on $(C,p,q)$ is
$\mu^L_{p,q}:\ T^L_{p,q} \to H^0(C, \omega_C)$
given by cup products. Abbreviate $\mu^L_{p,q} |_{T^L_{p,q}(S)}$ by $\mu^L_{p,q} |_S$; we call this the \emph{$S$-coupled Petri map}.
If $\mu^L_{p,q}$ is injective for every $L$, then we say that $(C,p,q)$ satisfies the \emph{(fully) coupled Petri condition}; if $\mu^L_{p,q} |_S$ is injective for all $L$, we say that $(C,p,q)$ satisfies the \emph{$S$-coupled Petri condition.}
\end{defn}

\begin{eg}
\label{eg:genus01}
If $C$ is genus $0$, then $T^L_{p,q}$ is trivial so $(C,p,q)$ always satisfies the coupled Petri condition. If $C$ is genus $1$, then $\dim T^L_{p,q}$ is equal to the number of pairs $(a,b)$ such that $L \cong \cO_C(ap+bq)$, and $(C,p,q)$ satisfies the coupled Petri condition if and only if $p-q$ is non-torsion.
\end{eg}

The $\{(0,0)\}$-coupled Petri map is the usual Petri map. A one-marked-point analog of $\mu^L_{p,q}$ is studied in \cite[Lemma 3.1]{cht}.
Corollary \ref{cor:mudelta} will show that $\cpb_d, \cqb_d$ is a versal pair at $[L]$ if and only if the $[0,d-2g+1]^2$-coupled Petri map of $L$ is injective. This paper proves

\begin{thm}
\label{thm:coupledPetri}
Let $S \subseteq \ZZ \times \ZZ$ be finite. A general $(C,p,q)$ satisfies the $S$-coupled Petri condition, and therefore a very general $(C,p,q)$ satisfies the fully coupled Petri condition.
\end{thm}

\subsection{Brill-Noether degeneracy loci}

We will apply the versality theorem to generalize basic theorems about $W^r_d(C)$ to the following degeneracy loci. Given $(C,p,q)$, an integer $d$, and a function $r: \NN^2 \to \NN$, consider the locus
\begin{equation}
\label{eq:degenFirstForm}
\{ [L] \in \Pic^d(C):\ h^0(C, L(-ap-bq)) \geq r(a,b) \mbox{ for all } a,b \in \NN^2 \}.
\end{equation}
By choosing $r(a,b) = \max(0, r_0+1 -a -b)$, this is the usual $W^{r_0}_d(C)$; other choices of function $r(a,b)$ encode information about the geometry of the map $C \to \PP H^0(C, L)^\vee$. 

Only certain functions $r(a,b)$ are useful the construction (\ref{eq:degenFirstForm}). In fact, the only functions where equality $r(a,b) = h^0(C, L(-ap-bq))$ is possible are described as follows.

\begin{defn}
A \emph{dot array} is a finite subset of $\NN \times \NN$ such that no two distinct elements of $\Pi$ are in the same row or the same column. When drawing a dot array, we index positions like entries in a matrix: $(a,b)$ is placed in row $a$ and column $b$, with $(0,0)$ in the upper-left corner.

The \emph{rank function} of a dot array is the function $\NN^2 \to \NN$ defined by 
\begin{equation}
\label{eq:rPi}
r^\Pi(a,b) = \# \{ (a',b') \in \Pi:\ a' \geq a \mbox{ and } b' \geq b \}.
\end{equation}
Given a dot array $\Pi$, the \emph{row sequence} is the set of first coordinates of elements of $\Pi$ sorted in increasing order and usually denoted $(a_0, \cdots, a_r)$, where $r = |\Pi|-1$, and the \emph{column sequence} is the sorted set of second coordinates, usually denoted $(b_0, \cdots, b_r)$.
\end{defn}

See \cite[$\S 1.2$]{fultonPragacz} for a nice discussion and many examples. However, note that we use slightly different index conventions since we think of $a,b$ as codimensions rather than dimensions. This better matches our intended applications, where $a,b$ will be vanishing orders. For example, we index beginning at $0$ and count ``towards the lower right'' rather than ``towards the upper left.''

\begin{defn}
\label{defn:bnDegen}
Let $(C,p,q)$ be a twice-marked curve of genus $g$, $d$ a positive integer, and $\Pi$ a dot array such that $|\Pi| \geq d+1-g$. Define set-theoretically
$$W^\Pi_d(C,p,q) = \{ [L] \in \Pic^d(C):\ h^0(C, L(-ap-bq)) \geq r^\Pi(a,b) \mbox{ for all } a,b \in \NN^2 \}.$$
Give this locus a scheme structure by interpreting each inequality with a Fitting ideal; see Section \ref{sec:bnDegen}. We call $W^\Pi_d(C,p,q)$ a \emph{Brill-Noether degeneracy locus.} 
\end{defn}

In fact, most of the inequalities in Definition \ref{defn:bnDegen} are redundant. Define
\begin{equation}
\label{eq:essPi}
\Ess(\Pi) = \{ (a,b) \in \NN^2: r^\Pi(a-1,b) = r^\Pi(a,b-1) = r^\Pi(a,b) > r^\Pi(a+1,b) = r^\Pi(a,b+1) \},
\end{equation}
and call this the \emph{essential set} of $\Pi$.
One may replace ``for all $a,b \in \NN^2$'' with ``for all $(a,b) \in \Ess(\Pi)$'' in Definition \ref{defn:bnDegen} without changing the scheme structure. See Section \ref{sec:flags}. It will also be convenient to take a slightly larger set than $\Ess(\Pi)$ for some arguments; define the \emph{essential rows} $\EssR(\Pi)$ and \emph{essential columns} $\EssC(\Pi)$ to be the images of $\Ess(\Pi)$ along the first and second projections to $\NN$, respectively. With these definitions, define
$$\widetilde{W}^\Pi_d(C,p,q) = \{ [L] \in \Pic^d(C): h^0(C,L(-ap-bq)) \geq r^\Pi(a,b) \mbox{ for all } (a,b) \in \EssR(\Pi) \times \EssC(\Pi) \}.$$
We illustrate the geometric significance of dot arrays and degeneracy loci in Figure \ref{fig:dotArrays}. This figure shows various examples of dot arrays with $3$ dots. The sketch shows a cartoon of a twice-marked curve with a line bundle $[L] \in \widetilde{W}^\Pi_d(C,p,q)$, immersed in $\PP^2$ by the complete linear series of $L$. The boxes of the set $\Ess(\Pi)$ are shaded.
\newcommand{\grid}{\foreach \x in {0,...,3} \draw (\x-0.5, 0.5) -- (\x-0.5,-3.5);
\foreach \y in {0,...,3} \draw (-0.5, 0.5-\y) -- (3.5, 0.5-\y);
}
\begin{figure}
\centering
\begin{tabular}{|cc|cc|cc|cc|}
\hline
\begin{tikzpicture}[scale=0.25]
\grid
\ess{0}{0}
\fdot{0}{2} \fdot{1}{1} \fdot{2}{0}
\end{tikzpicture}
& \begin{tikzpicture}[scale=0.2]
\coordinate (p) at (0,0);
\coordinate (q) at (4,1);
\node [fill=black, circle, inner sep=1pt, label=left:$p$] at (p) {};
\node [fill=black, circle, inner sep=1pt, label=right:$q$] at (q) {};
\draw (-0.5,-2) .. controls (-0.5,-1) .. (0,0) .. controls (1,2) and (3,-1) .. (4,1) .. controls (4.5,2) .. (4.5,3);
\end{tikzpicture}

&\begin{tikzpicture}[scale=0.25]
\grid
\ess{0}{0} \ess{3}{0}
\fdot{0}{2} \fdot{1}{1} \fdot{3}{0}
\end{tikzpicture}
& \begin{tikzpicture}[scale=0.2]
\coordinate (p) at (0,0);
\coordinate (q) at (2,2);
\node [fill=black, circle, inner sep=1pt, label=left:$p$] at (p) {};
\node [fill=black, circle, inner sep=1pt, label=right:$q$] at (q) {};
\draw (-2,-2) .. controls (2,-3) and (-2,3) .. (2,2) .. controls (4,1.5) and (4,0) .. (3,0);
\end{tikzpicture}

&\begin{tikzpicture}[scale=0.25]
\grid
\ess{0}{0} \ess{2}{0}
\fdot{0}{2} \fdot{2}{1} \fdot{3}{0}
\end{tikzpicture}
\begin{tikzpicture}[scale=0.2]
\coordinate (p) at (0,0);
\coordinate (q) at (4,0);
\node [fill=black, circle, inner sep=1pt, label=left:$p$] at (p) {};
\node [fill=black, circle, inner sep=1pt, label=right:$q$] at (q) {};
\draw (-3,-3) .. controls (-2,-3) and (-1,-3) .. (0,0) .. controls (-1,-3) and (3,-3) .. (4,0) .. controls (4.3333,1) and (3,2).. (2,1);
\end{tikzpicture}

&\begin{tikzpicture}[scale=0.25]
\grid
\fdot{0}{0} \fdot{1}{2} \fdot{2}{1}
\ess{1}{1} \ess{0}{0}
\end{tikzpicture}
& 
\begin{tikzpicture}[scale=0.2]
\coordinate (p) at (0,0);
\coordinate (q) at (0,0);
\node [fill=black, circle, inner sep=1pt, label=right:${p,q}$] at (q) {};
\draw (1.5,2) .. controls (1,1)  .. (0,0) .. controls (-3,-3) and (-3,3) .. (0,0) .. controls (1,-1) .. (1.5,-2);
\end{tikzpicture}\\

\multicolumn{2}{|c|}{generic} & \multicolumn{2}{c|}{flex point at $p$} & \multicolumn{2}{c|}{cusp at $p$} & \multicolumn{2}{c|}{node joining $p,q$}
\\\hline

\begin{tikzpicture}[scale=0.25]
\grid
\fdot{0}{1} \fdot{1}{0} \fdot{2}{2}
\ess{0}{0} \ess{2}{2}
\end{tikzpicture}
& 
\begin{tikzpicture}[scale=0.15]
\coordinate (p) at (0,0);
\coordinate (q) at (5,-5);
\node [fill=black, circle, inner sep=1pt, label=below left:${p}$] at (p) {};
\node [fill=black, circle, inner sep=1pt, label=above right:${q}$] at (q) {};
\draw[dashed] (-2,2) -- (7,-7);
\draw (-1,2) .. controls (-1,1.5) and (-0.5,0.5) .. (0,0) .. controls (1,-1) and (3,0) .. (4,-1) .. controls (5,-2) and (4,-4) .. (5,-5) .. controls (5.5,-5.5) and (6.5,-6) .. (7,-6);
\end{tikzpicture}

& \begin{tikzpicture}[scale=0.25]
\grid
\fdot{0}{0} \fdot{1}{1} \fdot{2}{2}
\ess{0}{0} \ess{1}{1} \ess{2}{2}
\end{tikzpicture}
& 
\begin{tikzpicture}[scale=0.6]
\coordinate (p) at (0,0);
\coordinate (q) at (0,0);
\node [fill=black, circle, inner sep=1pt, label=above:${p,q}$] at (p) {};
\draw (0.7,1) .. controls (0.5,0) and (0.1,0)  .. (0,0) .. controls (-0.5,0) and (-2,1) .. (-2,0) ..
controls (-2,-1) and (-0.5,0) .. (0,0) .. controls (0.1,0) and (0.5,0) .. (0.7,-1);
\end{tikzpicture}

& \begin{tikzpicture}[scale=0.25]
\grid
\fdot{3}{3} \fdot{2}{0} \fdot{0}{1}
\ess{0}{0}
\ess{2}{0}
\ess{3}{3}
\end{tikzpicture}
& 
\begin{tikzpicture}[scale=0.2]
\coordinate (p) at (0,0);
\coordinate (q) at (4,0);
\node [fill=black, circle, inner sep=1pt, label=above:${p}$] at (p) {};
\node [fill=black, circle, inner sep=1pt, label=above:${q}$] at (q) {};
\draw[dashed] (-2,0) -- (6,0);
\draw (-1,2) .. controls (-1,1) and (-1,0) .. (0,0) .. controls (-2,0) and (0,-3) .. (2,-3) .. controls (4,-3) and (6,0) .. (4,0) .. controls (2,0) and (2,1) .. (3,2);
\end{tikzpicture}

&\begin{tikzpicture}[scale=0.25]
\grid
\fdot{0}{0} \fdot{2}{1} \fdot{3}{2}
\ess{0}{0} \ess{2}{1} \ess{3}{2}
\end{tikzpicture}
& 
\begin{tikzpicture}[scale=0.2]
\coordinate (p) at (0,0);
\node [fill=black, circle, inner sep=1pt, label=above right:${p,q}$] at (p) {};
\draw (-1,2) .. controls (-2,1) and (-2,0) .. (0,0) .. controls (-3,0) and (0,-3) .. (2,-3) .. controls (4,-3) and (3,0) .. (0,0) .. controls (-2,0) .. (-4,-1);
\end{tikzpicture}
\\
\multicolumn{2}{|c|}{bitangent} & \multicolumn{2}{c|}{tacnode joining $p,q$} & \multicolumn{2}{c|}{node and flex on line} & \multicolumn{2}{c|}{cusp on tangent}
\\\hline
\end{tabular}
\caption{Examples of Brill-Noether degeneracy loci, with $|\Pi| = 3$. The shaded boxes indicate $\Ess(\Pi)$.}
\label{fig:dotArrays}
\end{figure}
For the ``generic'' dot array of size $r+1$, $\Pi = \{ (0,r), (1,r-1), \cdots, (r,0) \}$, $W^\Pi_d(C,p,q) = W^r_d(C)$, $\Ess(\Pi) = \{(0,0)\}$, and $\widetilde{W}^\Pi_d(C,p,q) = W^r_d(C) \backslash W^{r+1}_d(C)$. The choice of marked points is immaterial in this case.

The expected dimension of $W^\Pi_d(C,p,q)$ is the following number. 
\begin{eqnarray*}
\rho_g(d, \Pi) &=& g - (r+1)(g-d+r)
 - \sum_{i=0}^r (a_i-i) - \sum_{i=0}^r (b_i-i) \\
&&- \# \{ (a,b), (a',b') \in \Pi: a < a' \mbox{ and } b < b'\}.
\end{eqnarray*}
Here $a_0, \cdots, a_r$ and $b_0, \cdots, b_r$ denote the row and column sequences, and $r = |\Pi| -1$.
Dot patterns $\Pi \subset \NN^2$ with $| \Pi | \geq d+1-g$ are in bijection with a type of permutation of $\ZZ$ that we call \emph{$(d,g)$-confined permutations} and describe in Section \ref{ss:dgConfined}. The $(d,g)$-confined permutation $\pi$ associated to $\Pi$ has the feature that $\omega_{d-g} \pi$ has finite length, and in fact this length is the expected codimension: $\rho_g(d,\Pi) = g - \ell(\omega_{d-g} \pi)$. 

\begin{thm}
\label{thm:bnDegen}
Fix positive integers $g,d$ and an algebraically closed field $k$. Let $(C,p,q)$ be a general twice-marked smooth curve of genus $g$. For any dot array $\Pi$ such that $|\Pi| \geq d+1-g$, $W^\Pi_d(C,p,q)$ is nonempty if any only if $\rho_g(d,\Pi) \geq 0$. If it is nonempty, it has pure dimension $\rho_g(d, \Pi)$, and the open subscheme $\widetilde{W}^\Pi_d(C,p,q)$ is smooth and dense.

Let $\pi$ be the $(d,g)$-confined permutation associated to $\Pi$, and let $R(\omega_{d-g} \pi)$ denote the set of reduced words for the permutation $\omega_{d-g} \pi$ (or equivalently, saturated chains from the identity to $\omega_{d-g} \pi$ in the Bruhat order). Then the Chow class of $W^\Pi_d(C,p,q)$ is
$$\left[ W^\Pi_d(C,p,q) \right] = \frac{ |R(\omega_{d-g} \pi)| }{\ell(\omega_{d-g} \pi )!} \Theta^{\ell(\omega_{d-g} \pi)}.$$
\end{thm}

We will use the versality theorem to prove Theorem \ref{thm:bnDegen}. This analysis occupies Section \ref{sec:bnDegen}.

\subsection{Related work on Brill-Noether theory of marked curves}
\label{ss:bnBackground}
The present work builds upon a number of previous works on Brill-Noether theory of marked curves. The theorems related to this work fall under three types: dimension theorems, smoothness theorems, and enumerative theorems; all three problems can be studied for curves of various numbers of marked points. One aspect of classical Brill--Noether theory we do not address is \emph{connectedness}; for unmarked curves, connectedness of $W^r_d(C)$ follows from the main theorem of \cite{flConnectedness}, which does not apply in this situation. It is possible that a degeneration proof, along the lines of Osserman's proof of the connectedness of $W^r_d(C)$ \cite{ossermanConnectedness}, may work.

Eisenbud and Harris \cite{eh86} proved an extended form of the classical dimension theorem to curves with $n > 0$ marked points, with a characteristic $0$ hypothesis. This extended dimension theorem is false in positive characteristic when $n > 2$. However, the theorem is true in all characteristics when $n \leq 2$; a simple proof appears in \cite{ossSimpleBN}. The question of smoothness is considered, for $n=1$, in \cite{cht}, via a Petri map. In the case $n = 2$ and in all characteristics, a smoothness theorem was proved in \cite{cop}, by degeneration rather than using a Petri map. A smoothness theorem for the $n=2$ case via the injectivity of a Petri map was proved independently in \cite{teixidorBNTwo}, which appeared on arXiv five months after this paper. Our smoothness result departs from these results in part by considering not just ramification conditions at the two marked points, but also conditions on the interaction between the two marked points.

The enumerative geometry of Brill-Noether varieties began in Castelnuovo's calculation in 1889 in the case $\rho=0$; a nice summary of this argument and how it has been adapted to other situations is in \cite[\S 5A]{hm}. Modern approaches have typically used techniques from the theory of degeneracy loci, beginning with the work of Kempf \cite{kempfSchubert} and Kleiman-Laksov \cite{kleimanLaksov1,kleimanLaksov2}, who computed the intersection theoretic class $[W^r_d(C)]$ via the Porteous formula. An analogous computation for curves with one marked point may be deduced from the degeneracy locus formulas in \cite{kempfLaksov73}; this computation may be found, for example, in \cite[\S 4]{nonprimitive}, and an analogous theorem for curves with a \emph{moving} marked point can be found in \cite{farkasTarasca}. The case of two marked points with imposed ramification was studied in detail by Anderson, Chen, and Tarasca \cite{actKclasses}, who used degeneracy locus formulas to obtain not just the intersection class of the Brill-Noether loci but classes in K-theory as well. A recent paper \cite{actMotivic} also considers the case of one marked point from a motivic point of view. The enumerative aspects of this paper, developed in Section \ref{sec:bnDegen}, follow similar methods to those of \cite{actKclasses}. The principle difference in this paper is that while \cite{actKclasses} considers degeneracy loci indexed by $321$-avoiding permutations, imposing conditions on combinations of marked points requires us to consider degeneracy loci indexed by arbitrary permutations (of finite length).

The main innovations of the present work are the versality point of view, and the generalization from vanishing conditions at individual points to conditions imposed on combinations of the marked points.
For example, while vanishing conditions may impose a cusp, flex point, or higher ramification condition at a single marked point, the approach in this paper may also impose nodes, bitangents, and other conditions concerning the interaction of the two marked points.

Finally, we note that there are two points of view on classical Brill--Noether theory: one may consider loci $W^r_d(C)$ of special line bundles, or parameter spaces $G^r_d(C)$ of linear series. We focus herein on the first point of view.
Although the details are not developed in this paper, the versality theorem should also be able to provide new proofs of the main results of \cite{cop} about the geometry of the varieties $G^r_d(C, (p, a_\sbu), (q, b_\sbu))$
parameterizing linear series with prescribed vanishing orders.
This potential application is also discussed in \cite[$\S$6]{cpRR} and \cite[$\S$4]{liRBS}.
These varieties have the form of \emph{relative Richardson varieties}, in the sense of \cite{cpRR}. However, the results of \cite{cpRR} on such varieties work with pairs of \emph{complete} flags, so they do not immediately apply in this situation.
Assuming that the results of \cite{cpRR} generalize to partial flags, they would quickly imply the results of \cite{cop} about the smooth locus of a twice-pointed Brill-Noether variety, and also imply that the image of the projection
$G^r_d(C, (p, a_\sbu), (q, b_\sbu)) \to \Pic^d(C)$
is equal to a Brill-Noether degeneracy locus of the type discussed in this paper. Furthermore, the techniques of \cite{liRBS} provide natural resolutions of singularities for twice-pointed Brill-Noether varieties, via \emph{relative Bott--Samelson varieties}; see \cite[Conjecture 4.4]{liRBS}.

\subsection{Conventions}
\label{ss:conventions}

By a \emph{smooth curve} we mean a smooth proper geometrically connected curve over a field $k$ of any characteristic.
An \emph{arithmetic surface} is a proper flat scheme over a discrete valuation ring whose generic fiber is a smooth curve.
A \emph{point} refers to a $k$-valued point unless otherwise stated. A \emph{twice-marked smooth curve} $(C,p,q)$ is a smooth curve with two distinct marked points. For a twice-marked curve $(C,p,q)$, we denote the punctured curve $C \backslash \{p,q\}$ by $C^\ast$, where the marked points $p,q$ will be clear from context.

To simplify our notation, intervals $[a,b], (-\infty, b], [a, \infty)$ will always refer to sets of integers $[a,b] \cap \ZZ, (\infty, b] \cap \ZZ, [a, \infty) \cap \ZZ$, and the symbol $\NN$ will denote the set of \emph{nonnegative} integers.

A \emph{(partial) flag} in a vector bundle $\cH$ (or vector space) is a sequence of nested subbundles (or subspaces) $\cP^{a_0} \supset \cdots \supset \cP^{a_\ell}$, where the subscript always denotes corank (i.e. the rank of $\cH / \cP^a$ is $a$). The set $A = \{a_0, \cdots, a_\ell\}$ is called the set of \emph{coranks} of the flag. A \emph{complete flag} is a flag with set of coranks $[0,n]$, where $n$ is the rank of $\cH$. 

If $\alpha: \ZZ \xrightarrow{\sim} \ZZ$ is a permutation fixing all but finitely many $n \in \ZZ$, then the \emph{length} of $\alpha$ is $\ell(\alpha) = \# \{ (a,b) \in \ZZ: a < b \mbox{ and } \alpha(a) > \alpha(b) \}$. The length is also the minimum length of a factorization of $\alpha$ into simple transpositions; such a factorization is called a \emph{reduced word} for $\alpha$. 

Denote by $\omega_m$ the descending permutation $\omega_m(n) = m-n$, which we will sometimes also regard as a permutation of $[0,m-1]$ rather than all of $\ZZ$. 

\subsection{Outline of the paper}
Section \ref{sec:degen} carries out the degeneration argument needed to prove Theorem \ref{thm:coupledPetri}. Section \ref{sec:flags} gives some preliminary material about relative positions of flags. Section \ref{sec:versality} is concerned with first-order deformations of a pair of flags, and builds the link between versality of Brill-Noether flags and the coupled Petri map. Section \ref{sec:proofs} completes the proofs of Theorems \ref{thm:versality} and \ref{thm:coupledPetri}. Section \ref{sec:bnDegen} applies the versality theorem to the analysis of Brill-Noether degeneracy loci and proves Theorem \ref{thm:bnDegen}.

\begin{rem}
This paper was substantially shortened during the editorial process. If some cases, more detailed proofs and explanations can be found in arXiv version $2$.
\end{rem}

\section{Degeneration to chains of twice-marked curves}
\label{sec:degen}

This section is largely independent of the rest of the paper; its object may be summarized by the slogan: ``twice-marked curves satisfying the coupled Petri condition can be chained together.'' The approach is based on \cite{ehgp}, with two main adaptations. First, we degenerate to a chain of genus-$1$ curves, rather than a flag curve; this is similar to the approach taken in \cite{welters}, and is analogous to the tropical proof of the Gieseker--Petri theorem \cite{jpGP}. In fact, we formulate our degeneration in a more flexible way. Any chain of smooth curves, all satisfying the coupled Petri condition when the attachment points are marked, will do. Second, we work throughout with sections of line bundles on twice-punctured curves; this adds the flexibility to consider the fully coupled Petri condition all at once, and also simplifies the argument somewhat.
\begin{figure}
\centering
\begin{tikzpicture}[thick]
\coordinate (p) at (0,0);
\coordinate (q) at (0,-5);
\coordinate (p1) at (5,0);
\coordinate (p2) at (5,-1);
\coordinate (p3) at (5,-2);
\coordinate (p4) at (5,-3);
\coordinate (p5) at (5,-4);
\coordinate (q5) at (5,-5);
\coordinate (eta) at (0,-7);
\coordinate (zero) at (5,-7);

\node [fill=black, circle, inner sep=2pt, label=above left:$P_\eta$] at (p) {};
\node [fill=black,circle, inner sep=2pt, label=above left:$Q_\eta$] at (q) {};
\node [fill=black,circle, inner sep=2pt, label=above left:$p_1$] at (p1) {};
\node [fill=black,circle, inner sep=2pt, label=left:{$q_1=p_2$}] at (p2) {};
\node [fill=black,circle, inner sep=2pt, label=left:{$q_2 = p_3$}] at (p3) {};
\node [fill=black,circle, inner sep=2pt, label=left:{$q_3 = p_4$}] at (p4) {};
\node [fill=black,circle, inner sep=2pt, label=left:{$q_{\ell-1} = p_\ell$}] at (p5) {};
\node [fill=black,circle, inner sep=2pt, label=below left:$q_\ell$] at (q5) {};

\node [fill=black,circle, inner sep=2pt, label=below:$\eta$] at (eta) {};
\node [fill=black,circle, inner sep=2pt, label=below:$0$] at (zero) {};

\draw [shorten <= -0.5cm, shorten >= -0.5cm] (p) to[in=-170, out=-10] node[midway, above] {$P = A_0$} (p1);
\draw [shorten <= -0.5cm, shorten >= -0.5cm] (q) to[in=-170, out=-10] node[midway, above] {$Q = A_{\ell+1}$} (q5);

\draw [shorten <= -0.5cm, shorten >= -0.5cm] ([xshift=-0.05cm]p) to[in=45, out=-135] node[midway, right] {$\cC_\eta$} ([xshift=0.05cm]q);

\draw [shorten <= -0.2cm, shorten >= -0.2cm] (p1) to[out=-135, in=135] node[midway,right] {$A_1$} (p2);
\draw [shorten <= -0.2cm, shorten >= -0.2cm] (p2) to[out=-135, in=135] node[midway,right] {$A_2$} (p3);
\draw [shorten <= -0.2cm, shorten >= -0.2cm] (p3) to[out=-135, in=135] node[midway,right] {$A_3$} (p4);
\draw [shorten <= -0.2cm, shorten >= -0.2cm] (p5) to[out=-135, in=135] node[midway,right] {$A_\ell$} (q5);
\path ([yshift=0.25cm]p4) -- node[midway] {$\vdots$} (p5);

\draw[->] (2.5, -5.4) -- (2.5, -6.8);

\draw [shorten <= -0.4cm, shorten >= -0.4cm] (eta) -- node[midway, below] {$\Spec R$} (zero);

\end{tikzpicture}
\caption{The arithmetic surface $\cC \to \Spec R$ in Situation \ref{sit:chain}.}
\label{fig:chain}
\end{figure}

\begin{sit}
\label{sit:chain} 
Let $R$ be a discrete valuation ring, with residue field $k$, uniformizer $t$, and fraction field $K$. Assume that $k$ is algebraically closed.
Let $\pi: \cC \to \Spec R$ be a regular arithmetic surface, with two disjoint sections $P, Q$. We sometimes also denote $P$ by $A_0$ and $Q$ by $A_{\ell+1}$. Assume that the special fiber $\cC_k$ is a chain of $\ell$ smooth curves $A_1 \cup \cdots \cup A_\ell$, where $A_n$ and $A_{n+1}$ meet nodally at a point that we denote by both $q_n$ and $p_{n+1}$.
Assume $P$ meets the special fiber at a point $p_1 \in A_1$ distinct from $q_1$, and $Q$ meets the special fiber at a point $q_\ell \in A_\ell$ distinct from $p_\ell$.

Denote by $\cC^\ast$ the complement of $P \cup A_1 \cup \cdots \cup A_\ell \cup Q = A_0 \cup \cdots \cup A_{\ell+1}$, and by $\cU_n$ the complement of $A_0 \cup \cdots \cup A_{n-1} \cup A_{n+1} \cup \cdots \cup A_{\ell+1}$. Also denote by $A_n^\ast$ the punctured smooth curve $A_n \backslash \{p_n, q_n\}$ (for $1 \leq n \leq \ell$).
The generic fiber $\cC \times_R \Spec K$ will be denoted $C_\eta$, and the geometric generic fiber $\cC \times_R \Spec \overline{K}$ will be denoted by $C_{\overline{\eta}}$. 
If $\cL,\cM$ are line bundles on $\cC$, we will write $L_n, M_n$ to denote the restrictions of these line bundles to $A_n$.
\end{sit}

\begin{thm}
\label{thm:petriSmoothing}
In Situation \ref{sit:chain}, if each $(A_n, p_n, q_n)$ satisfies the coupled Petri condition, then $(C_\be, P_\be, Q_\be)$ satisfies the coupled Petri condition.
\end{thm}

It will simplify notation slightly and add no additional difficulty to work in a slightly more general situation. If $\cL, \cM$ are line bundles on $\cC$, denote by $\mu$ the cup product map
$$\mu: H^0(\cC^\ast, \cL) \otimes_R H^0(\cC^\ast, \cM) \to H^0(\cC^\ast, \cL \otimes \cM).$$
We will define a notion of coupled tensors for this multiplication map and prove in Theorem \ref{thm:icdegen} an analog of Theorem \ref{thm:petriSmoothing} in this setting, from which we will deduce Theorem \ref{thm:petriSmoothing}.

\subsection{Valuation of tensors}
\label{ss:valTensor}

In Situation \ref{sit:chain}, the divisors $P = A_0, A_1, \cdots, A_\ell, A_{\ell+1} = Q$ determine valuations $\nu_0, \cdots, \nu_{\ell+1}$ on the function field of $\cC$. If $\cL$ is a line bundle on $\cC$, then $\nu_n$ also naturally determines a map $\nu_n: H^0( \cC^\ast, \cL) \to \ZZ \cup \{ \infty \}$ giving the vanishing order of a rational section along $A_n$. Following \cite{ehgp}, we extend $\nu_n$ to tensors, as follows. If $\cL,\cM$ are two line bundles on $\cC$, then we define
$$\nu_n: H^0(\cC^\ast, \cL) \otimes_R H^0(\cC^\ast, \cM) \to \ZZ \cup \{ \infty \}$$
by setting $\nu_n(\rho)$ to the maximum, taken over all expansions $\rho = \sum \sigma_i \otimes \tau_i$ into rank-$1$ tensors, of $\min_i\{ \nu_n(\sigma_i) + \nu_n(\tau_i) \}$. Importantly, $\nu_n(\rho)$ is \emph{not} always equal to $\nu_n( \mu(\rho))$.
In like fashion, we extend the valuations $\nu_{p_n}, \nu_{q_n}$ to maps $H^0(A_n^\ast, L_n) \otimes_k H^0(A_n^\ast, M_n) \to \ZZ \cup \{ \infty \}$.

For $n \in [1,\ell]$, $\nu_n |_R$ is the discrete valuation on $R$, so for all $a \in \ZZ$,
$$\big\{ \rho \in H^0(\cC^\ast, \cL) \otimes_R H^0(\cC^\ast, \cM):\ \nu_n(\rho) \geq a \big\} = t^a H^0(\cU_n, \cL) \otimes_R H^0(\cU_n, \cM).$$
On the other hand, if $n \in \{0, \ell+1\}$, then $\nu_n$ induces the trivial valuation on $R$. In general, 
\begin{equation}
\label{eq:nuMult}
\nu_n( t^a \rho ) = \begin{cases}
a + \nu_n(\rho) & \mbox{ if } 1 \leq n \leq \ell\\
\nu_n(\rho) & \mbox{ if } n \in \{0,\ell+1\}.
\end{cases}
\end{equation}

There is a restriction map
$H^0(\cU_n, \cL) \otimes_R H^0(\cU_n, \cM) \to H^0(A_n^\ast, L_n) \otimes_k H^0(A_n^\ast, M_n),$
for all $1 \leq n \leq \ell$, the kernel of which consists of all tensors divisible by $t$, i.e. all $\rho$ with $\nu_n(\rho) \geq 1$. We denote the image of a tensor $\rho$ under this homomorphism by $\rho |_{A_n^\ast}$; this is nonzero if $\nu_n(\rho) = 0$.

\begin{prop}
\label{prop:nuRestriction}
Let $\rho \in H^0(\cC^\ast, \cL) \otimes_R H^0(\cC^\ast, \cM)$, and let $n \in [1,\ell]$. If $\nu_n(\rho) = 0$, then the restriction $\rho |_{A_n^\ast}$ is nonzero. It satisfies
$\nu_{p_n} \left( \rho |_{A_n^\ast} \right) \geq \nu_{n-1}(\rho)$\; and\; $\nu_{q_n} \left( \rho |_{A_n^\ast} \right) \geq \nu_{n+1}(\rho).$
\end{prop}

\subsection{Adapted bases}
Proposition \ref{prop:nuRestriction} is plausible enough, but it requires a bit of algebra to prove.
Namely, we make use of a special type of basis for finite-dimensional subspaces of sections of line bundles on $\cC^\ast$. These bases play the role of the ``compatible bases'' used in \cite[Lemma 1.2]{ehgp}.

\begin{defn}
\label{def:adapted}
In Situation \ref{sit:chain}, let $\cL$ be a line bundle on $\cC$, and suppose $n \in [0,\ell+1]$. Let $S \subseteq H^0(\cC, \cL)$, and let $V = \operatorname{span} S$. The set $S$ is called \emph{adapted to $A_n$} if it is linearly independent (over $K$), and for all $a \in \ZZ$, the set $\{ c \sigma:\ c \in K,\ \sigma \in S,\ \nu_n(c \sigma) \geq a \}$ generates the $R$-module $V_a = \{v \in V:\ \nu_n(v) \geq a\}$.
\end{defn}

Adapted bases are convenient for evaluating valuations of sections, since
\begin{equation}
\label{eq:adaptedValuation}
\nu_n\left( \sum c_i \sigma_i \right) = \min \{ \nu_n(c_i \sigma_i) \}.
\end{equation}

\begin{lemma}
\label{lem:adaptedTwo}
Fix a finite-dimensional subspace $V \subseteq H^0(\cC^\ast, \cL)$, and let $m,n \in [0,\ell+1]$. There exists a basis of $V$ that is adapted to both $A_m$ and $A_n$.
\end{lemma}
\begin{proof}
    Filtering by $\nu_m$, one can obtain a basis of $V$ adapted to $A_m$. Given two indices $m,n$, we may choose bases adapted to $A_m$ and $A_n$ individually, and form the change of basis matrix from one to the other. By Gaussian elimination, being careful to multiply by constants that do not lower valuation, we may modify the two bases until this matrix has only one nonzero element in each row or column.
\end{proof}

\begin{lemma}
\label{lem:twoNu}
Suppose $\rho \in H^0(\cC^\ast, \cL) \otimes_R H^0(\cC^\ast, \cM)$. If $\nu_m(\rho) \geq a$ and $\nu_n(\rho) \geq b$, then there exist $\sigma_1, \cdots, \sigma_N \in H^0(\cC^\ast, \cL)$ and $\tau_1, \cdots, \tau_N \in H^0(\cC^\ast, \cM)$ such that 
\begin{enumerate}
\item $\rho = \sum \sigma_i \otimes \tau_i$, 
\item $\nu_m(\sigma_i) + \nu_m(\tau_i) \geq a$ for all $i$, and
\item $\nu_n(\sigma_i) + \nu_n(\tau_i) \geq b$ for all $i$.
\end{enumerate}
\end{lemma}

\begin{proof}
By definition of $\nu_m(\rho), \nu_n(\rho)$, there exist two expansions $\rho = \sum \alpha_k \otimes \beta_k = \sum \gamma_k \otimes \delta_k$ such that $\nu_m(\alpha_k) + \nu_m(\beta_k) \geq a$ and $\nu_n(\gamma_k) + \nu_n(\delta_k) \geq b$. 
We may choose bases $\{\sigma_i\}$ of the span of all $\alpha_i$ and $\gamma_i$, and $\{\tau_j\}$ of the span of all $\beta_i$ and $\delta_i$, both adapted to both $A_m$ and $A_n$. Expanding in these bases gives $\rho = \sum f_{i,j}\ \sigma_i \otimes \tau_j$ where $\nu_m(f_{i,j}) + \nu_m(\sigma_i) + \nu_m(\tau_i) \geq a$ for all $i,j$.
\end{proof}

We now have the tools to prove Proposition \ref{prop:nuRestriction}.

\begin{proof}[Proof of Proposition \ref{prop:nuRestriction}]
Suppose that $1 \leq n \leq \ell$ and $\rho \in H^0(\cC^\ast, \cL) \otimes_R H^0(\cC^\ast, \cM)$ satisfies $\nu_n(\rho) = 0$. The claim that $\rho|_{A_n^\ast}$ is nonzero was verified in the paragraph before the proposition.
By Lemma \ref{lem:twoNu}, there exists an expansion
$\rho = \sum \sigma_i \otimes \tau_i$
such that $\nu_{n-1}(\sigma_i) + \nu_{n-1}(\tau_i) \geq \nu_{n-1}(\rho)$ and $\nu_n(\sigma_i) + \nu_n(\tau_i) \geq 0$ for all $i$. Since $\nu_n(t) = 1$, we may assume that $\nu_n(\sigma_i), \nu_n(\tau_i) \geq 0$ for all $i$. Thus we have well-defined restrictions $\sigma_i |_{A_n^\ast} \in H^0(A_n^\ast, L_n)$ and $\tau_i |_{A_n^\ast} \in H^0(A_n^\ast, M_n)$, and 
$\rho |_{A_n^\ast} = \sum (\sigma_i |_{A_n^\ast}) \otimes (\tau_i |_{A_n^\ast}).$
The restriction of $\cO_{\cC}(A_{n-1})$ to $A_n$ is equal to $\cO_{A_n}(p_n)$, and thus $\nu_{p_n}(\sigma_i |_{A_n^\ast}) \geq \nu_{n-1}(\sigma_i)$ and $\nu_{p_n}(\tau_i |_{A_n^\ast}) \geq \nu_{n-1}(\tau_i)$. Therefore $\nu_{p_n} (\rho |_{A_n^\ast}) \geq \min_i \left( \nu_{n-1}(\sigma_i) + \nu_{n-1}(\tau_i) \right) = \nu_{n-1}(\rho)$, as desired. The bound $\nu_{q_n}(\rho |_{A_n^\ast}) \geq \nu_{n+1}(\rho)$ follows by replacing $\nu_{n-1}$ with $\nu_{n+1}$.
\end{proof}

\subsection{Injective coupling}
\label{ss:injectiveCoupling}

This subsection establishes our basic inductive result. Suppose that $C \to \Spec k$ is a smooth curve over a field,
$p,q$ are distinct rational points, and $L,M$ are two line bundles on $C$. Define
$$T^{L,M}_{p,q} =  \sum_{a,b \in \ZZ} H^0(C, L(-ap-bq)) \otimes_k H^0(C, M(ap+bq)),$$
and define the \emph{coupled multiplication map} to be the combination of cup product maps
\begin{equation}
\label{eq:coupledMult}
\mu^{L,M}_{p,q}:\
T^{L,M}_{p,q}
 \to H^0(C, L \otimes M),
\end{equation}
We say that two line bundles $L,M$ have \emph{injective coupling} if $\mu^{L,M}_{p,q}$ is injective.

\begin{thm}
\label{thm:icdegen}
In Situation \ref{sit:chain}, let $\cL, \cM$ be line bundles. If $L_n(-p_n)$ and $M_n$ have injective coupling on $(A_n,p_n,q_n)$ for all $1 \leq n \leq \ell$, then $L_\eta(-P_\eta), M_\eta$ have injective coupling on $(C_\eta, P_\eta, Q_\eta)$.
\end{thm}

\begin{proof}
Suppose that $\rho \in \ker \mu^{L_\eta,M_\eta}_{P_\eta,Q_\eta}$ is nonzero and satisfies $\nu_P(\rho) \geq 1$. We will demonstrate that $\nu_Q(\rho) < 0$.
For each $1 \leq n \leq \ell$, let $\rho_n = t^{-\nu_n(\rho)} \rho$. Proposition \ref{prop:nuRestriction} implies that $\rho_n |_{A_n^\ast}$ is nonzero and satisfies $\nu_{p_n}( \rho_n |_{A_n^\ast} ) \geq \nu_{n-1}(\rho_n)$ and $\nu_{q_n}( \rho_n |_{A_n^\ast} ) \geq \nu_{n+1}(\rho_n)$. Injective coupling on $A_n$ implies that either $\nu_{n-1}(\rho_n) < 1$ or $\nu_{n+1}(\rho_n) < 0$.
By Equation \ref{eq:nuMult}, for $1 \leq n \leq \ell-1$,

$$\nu_{n-1}(\rho_n) = \begin{cases} \nu_P(\rho) & \mbox{ if } n = 1 \\ \nu_{n-1}(\rho) - \nu_n(\rho) & \mbox{ if } n \geq 2 \end{cases},
\mbox{ and } 
\nu_{n+1}(\rho_n) = \begin{cases} \nu_Q(\rho) & \mbox{ if } n = \ell \\ \nu_{n+1}(\rho) - \nu_n(\rho) & \mbox{ if } n \leq \ell-1 \end{cases}.$$

If $\nu_P(\rho) \geq 1$, then $\nu_{n+1}(\rho) < \nu_n(\rho)$ for $1 \leq n \leq \ell-1$ by induction, and $\nu_Q(\rho) < 0$. 
\end{proof}

The following lemma is needed to apply Theorem \ref{thm:icdegen} to the coupled Petri map.

\begin{lemma}
\label{lem:extraQHarmless}
Let $(C,p,q)$ be a twice-marked smooth curve of any genus, and $L$ any line bundle on $C$. The line bundles $L$ and $\omega_C \otimes L^\vee$ have injective coupling if and only if $L(-p)$ and $\omega_C(p+q) \otimes L^\vee$ have injective coupling.
\end{lemma}
\begin{proof}
By Riemann--Roch, for all $a,b$ either $H^0(C, L(-ap-(b-1)q)) = H^0(C,L(-ap-bq))$ or $H^0(C, \omega_C \otimes L^\vee(ap + bq)) = H^0(C, \omega_C \otimes L^\vee(ap + (b-1) q))$. Therefore $T^{L, \omega_C \otimes L^\vee}_{p,q} = T^{L(q), \omega_C \otimes L^\vee}_{p,q}$. By re-indexing the sum in the definition, $T^{L, \omega_C \otimes L^\vee}_{p,q} = T^{L(-p), \omega_C(p+q) \otimes L^\vee}_{p,q}$.
\end{proof}

\begin{proof}[Proof of Theorem \ref{thm:petriSmoothing}]
In Situation \ref{sit:chain}, assume that each $(A_n, p_n, q_n)$ satisfies the coupled Petri condition. Let $L$ be any line bundle on $\cC_{\be}$; we will prove that $\mu^L_{P_\be, Q_\be}$ is injective.

First consider the case where $L$ is defined over $K$, so that $L$ may be regarded as a line bundle on $C_\eta$. Since $\cC$ is regular, this line bundle may be extended (non-uniquely) to a line bundle $\cL$ on $\cC$.
Let $\cM = \omega_{\cC/R} (P+Q) \otimes \cL^\vee$. Then $M_n = \omega_{A_n}(p_n+q_n) \otimes L_n^\vee$, by adjunction \cite[Theorem 9.1.37]{liu}. 
By assumption, $L_n$ and $\omega_{A_n} \otimes L_n^\vee$ have injective coupling; by Lemma \ref{lem:extraQHarmless} so do $L_n(-p_n)$ and $M_n$. 
Now Theorem \ref{thm:icdegen} implies that $L(-P_\eta)$ and $M = \omega_{\cC_\eta}(P_\eta+Q_\eta) \otimes L^\vee$ have injective coupling on $C_\eta$, and Lemma \ref{lem:extraQHarmless} shows that $L$ and $\omega_{\cC_\eta} \otimes L^\vee$ do too. The same is true after $L$ is pulled back from $C_\eta$ to $C_\be$, since the global section functor commutes with flat base extension.

For the general case, note that any line bundle $L$ on $C_\be$ is defined over some finite extension $K' \supseteq K$. 
We may extend $R$ to a discrete valuation ring $R'$ with fraction field $K'$. 
After base extension to $R'$ and resolving singularities, we obtain $\cC' \to \Spec R'$ with special fiber given by replacing each node $p_n$ ($2 \leq n \leq \ell$) by a chain of rational curves; see e.g. \cite[Corollary 3.25]{liu}. 
Rational curves have injective coupling (Example \ref{eg:genus01}), so we may invoke the first case to conclude that $L$ has injective coupling on $C_\be$.
\end{proof}

\section{Flags and permutations}
\label{sec:flags}

This section collects, without proofs, some basic facts about how permutations are used to record the relative position of a pair of flags in a fixed vector space and to define degeneracy loci of flags within vector bundles. There are a variety of conventions in the literature, depending on whether flags are indexed by dimension or codimension, whether indexing begins at $0$ or $1$, and other choices, so it is useful to consolidate the facts that we need in our notation.

\begin{fact}
\label{fact:permBasis}
Let $P^\sbu, Q^\sbu$ be two flags in a vector space $H$, with sets of coranks $A$ and $B$ respectively, and let $n = \dim H$. There exists a unique permutation $\sigma$ of $[0,n-1]$ such that
\begin{enumerate}
\item For all $a \in A, b \in B$, $\dim P^a \cap Q^b = \# \{ i \in [0,n-1]: i \geq a \mbox{ and } \sigma(i) \geq b \}$,
\item for all $a \in [0,n-1]$ with $a > 0$ and $a \not\in A$, $\sigma(a) < \sigma(a-1)$, and 
\item for all $b \in [0,n-1]$ with $b >0$ and $b \not\in B$, $\sigma^{-1}(b) < \sigma^{-1}(b-1)$.
\end{enumerate} 
There exists a basis $\{v_0, \cdots, v_{n-1}\}$ of $H$ such that for all $a \in A, b \in B$, $\{v_i: i \geq a \}$ is a basis for $P^a$ and $\{v_i:\ \sigma(i) \geq b\}$ is a basis for $Q^b$. Therefore $\{ v_i:\ i \geq a, \sigma(i) \geq b \}$ is a basis for $P^a \cap Q^b$.
\end{fact}

The permutation $\sigma$ is called the \emph{permutation associated to $P^\sbu, Q^\sbu$}. We call any permutation of $[0,n-1]$ satisfying conditions (2) and (3) above \emph{compatible with coranks $A, B$}. We will call a basis as in Fact \ref{fact:permBasis} \emph{adapted to $P^\sbu, Q^\sbu$}. 

The association of a permutation to a pair of flags can also be reversed, to define \emph{degeneracy loci} for a pair of flags in a vector bundle. We follow the conventions of \cite[$\S 2$]{cpRR} here. Given a permutation $\sigma$ of $[0,n-1]$, define the \emph{rank function} $r^\sigma: \NN^2 \to \NN$ of $\sigma$ by
\begin{equation}
\label{eq:rsigma}
r^\sigma(a,b) = \# \{a' \in [0,n-1]: a' \geq a \mbox{ and } \sigma(a') \geq b \}.
\end{equation}
If $S$ is a scheme with a rank $n$ vector bundle $\cH$ and two complete flags $\cpb, \cqb$ in $\cH$, define
\begin{equation}
\label{eq:dsigmaComplete}
D_\sigma(\cpb; \cqb) = \big\{ x \in S:\ \dim (\cP^a)_x \cap (\cQ^b)_x \geq r^\sigma(a,b) \mbox{ for all } a,b \in [0,n-1] \big\}.
\end{equation}
The equation above is only a set-theoretic definition. To give a scheme-theoretic definition, translate each inequality into a bound on the rank of the bundle map $\cQ^b \to \cH / \cP^a$, trivialize this bundle map locally, and consider the scheme defined by all the minors of appropriate size in the matrix representation. Equivalently, one can bound the rank of the difference map $\cP^a \oplus \cQ^b \to \cH$.

Although there are a priori $d^2$ inequalities imposed in Equation \ref{eq:dsigmaComplete}, many of them are redundant. 
A minimal set of rank conditions is provided by the \emph{essential set}, defined in \cite{fultonSchubert}.
\begin{equation}
\label{eq:essSet}
\Ess(\sigma) = \{(a,b): 1 \leq a,b < n,\ \sigma(a-1) < b \leq \sigma(a) \mbox{ and } \sigma^{-1}(b-1) < a \leq \sigma^{-1}(b) \}.
\end{equation}
The inequalities in Equation \ref{eq:essSet} are equivalent to the equations and inequalities
\begin{equation}
\label{eq:essSetV2}
r^\sigma(a-1,b) = r^\sigma(a,b-1) = r^\sigma(a,b) > r^\sigma(a+1,b)= r^\sigma(a,b+1),
\end{equation}
which make the redundancy more clear (cf. Equation \ref{eq:essPi}).
By \cite[Lemma 3.10]{fultonSchubert}, the scheme defined in Equation \eqref{eq:dsigmaComplete} is the same if we consider only $(a,b) \in \Ess(\sigma)$. A pleasant consequence of this is that it is also possible to define $D_\sigma$ for \emph{partial} flags, provided that the flags have strata of all coranks mentioned in $\Ess(\sigma)$.

\begin{defn}
\label{def:dsigma}
Let $\cpb, \cqb$ be flags with coranks $A,B$ respectively, and let $\sigma$ be a permutation of $[0,n-1]$ such that $\Ess(\sigma) \subseteq A \times B$. Define
$$D_\sigma(\cpb; \cqb) = \{ x \in S: \dim (\cP^a)_x \cap (\cQ^b)_x \geq r^\sigma(a,b) \mbox{ for all } (a,b) \in \Ess(\sigma) \},$$
where the scheme structure is defined in the manner described above.
\end{defn}

Permutations of $[0,n-1]$ have a partial order, the \emph{Bruhat order}, in which $\sigma \leq \tau$ if and only if $r^\sigma \geq r^\tau$. Therefore $\tau \leq \sigma$ implies $D_{\tau}(\cpb; \cqb) \subseteq D_{\sigma}(\cpb; \cqb)$.

Denote by $\Fl(A; n)$ the flag variety parameterizing flags $V^\sbu$ of coranks $A$ in $k^n$. Taking $\cV^\sbu$ to be the tautological flag of coranks $A$ on $\Fl(A;n)$, and $F^\sbu$ to be a fixed flag in $k^n$ of coranks $B$, we denote
$$X_\sigma(F^\sbu) = D_\sigma(\cV^\sbu; F^\sbu) \subseteq \Fl(A; n),$$
where $\sigma$ is any permutation compatible with coranks $A, B$. These are \emph{Schubert varieties}, and they are prototypical of all degeneracy loci of pairs of flags. 

\begin{fact}
\label{fact:Schubert}
Fix a permutation $\sigma$ compatible with coranks $A,B$, and a fixed flag $F^\sbu$ of coranks $B$ in $k^n$. The Schubert variety $X_\sigma(F^\sbu)$ has codimension $\ell( \omega_{n-1} \sigma)$ in $\Fl(A; n)$. Its singular locus is a union of Schubert varieties $X_\tau(F^\sbu)$ for certain permutations $\tau \leq \sigma$. All such $\tau$ have $\ell(\omega_{n-1} \tau) \geq \ell(\omega_{n-1} \sigma) + 2$, so Schubert varieties are smooth in codimension $1$.
\end{fact}

If $X_{\tau}(F^\sbu) \subseteq X_{\sigma}(F^\sbu)_{\operatorname{sing}}$, we will say that \emph{$\tau$ is singular in $X_\sigma$}. A combinatorial description of those $\sigma, \tau$ for which $\tau$ is singular in $X_\sigma$ was conjectured in \cite{lakshmibai-sandhya-criterion} and proved independently by several groups \cite{billey-warrington-maximal, cortez-singularities, kassel-lascoux-reutenauer-singular, manivel-lieu}.

\section{Versality of a pair of flags}
\label{sec:versality}

This section gives a definition and some preliminary results about versality of flags, and applies these results to Brill-Noether flags. We follow the point of view of \cite{cpRR}, but do not assume the flags are complete, and restrict attention to pairs of \emph{two} flags. 
Throughout this section, let $S$ be a finite-type scheme over an algebraically closed field $k$ of any characteristic, with a rank $n$ vector bundle $\cH$. Suppose further that two (partial or complete) flags $\cpb, \cqb$ are chosen, with sets of coranks $A,B$ respectively. We will reserve the letters $a,b$ for elements of $A$ or $B$ respectively.

\begin{defn}
\label{defn:versal}
Denote by $f: \Fr(\cH) \to S$ the frame bundle of $\cH$ and by $\Fl(A; n)$ the flag variety parameterizing flags $F^\sbu$ of coranks $A$. 
The flags $\cpb, \cqb$ induce $p: \Fr(\cH) \to \Fl(A;n) \times \Fl(B;n)$.
The pair $\cpb, \cqb$ is \emph{versal} if $p$ is a smooth morphism.
\end{defn}

\subsection{First-order deformations of a pair of flags}
Fix a $k$-point $x \in S$. For simplicity, denote by $H, P^\sbu, Q^\sbu$ the fibers of $\cH, \cpb, \cqb$ over $x$. Define
\begin{equation}
\label{eq:Mspace}
M_x = \operatorname{coker} \left( \End H \xrightarrow{\Delta} \End H / \Fix P^\sbu \oplus \End H / \Fix Q^\sbu \right).
\end{equation}
where $\Fix P^\sbu$ denotes $\{ \phi \in \End H:\ \phi(P^a) \subseteq P^a \mbox{ for all $a \in A$}\}$, 
and $\Delta$ is the diagonal map. 

Fix a point $y \in f^{-1}(x)$. This amounts to an isomorphism $k^n \xrightarrow{\sim} H$, which induces
$$T_{p(y)} \Fl(A; n) \times \Fl(B; n) \xrightarrow{\sim} \End H / \Fix P^\sbu \oplus \End H / \Fix Q^\sbu.$$
Using this identification we define a quotient map $q_y: T_{p(y)} \Fl(A;n) \times \Fl(B;n) \to M_x$. The kernel of $q_y$ is isomorphic to the image of the diagonal map from $\End H$, and may be identified with the tangent space to the $\GL_n(k)$-orbit of $p(y)$. The following lemma follows from the discussion in \cite[$\S$3]{cpRR}, which generalizes with minor notation changes to the case of incomplete flags. 

\begin{lemma}
\label{lem:deltax}
There is a a linear map $\delta_x: T_x S \to M_x$, independent of the choice of $y$, such that the following diagram commutes. In this diagram, the rows are exact, $\GL_n(k)$ is identified with the fiber $f^{-1} (x)$ and the map to $T_y \Fr(\cH)$ is the differential of the inclusion.
\begin{center}
\begin{tikzcd}
0 \ar[r] & T_{\operatorname{id}} \GL_n(k) \ar[r] \ar[d, equals] & T_y \Fr(\cH) \ar[d, "dp_y"] \ar[r,"df_y"]& T_x S \ar[r] \ar[d,"\delta_x"] & 0 \\
0 \ar[r] & \End H \ar[r, "\Delta"] & T_{p(y)} \Fl(A; d) \times \Fl(B; d) \ar[r,"q_y"] & M_x \ar[r] & 0\\
\end{tikzcd}
\end{center}
The flag pair $\cP^\sbu, \cQ^\sbu$ is versal if and only if $S$ is smooth over $k$ and $\delta_x$ is surjective for all $x \in S$.
\end{lemma}

We require a more concrete description of the map $\delta_x$ suitable to study Brill-Noether flags. Let $v \in T_x S$, and regard $v$ as a morphism $v: \Spec k[\eps] \to S$, where $k[\eps]$ is the ring of dual numbers. This gives an extension $0 \to H \to v^\ast \cH \to H \to 0$, where the first map is multiplication by $\eps$ and the second is quotient modulo $\eps$. A choice of $y \in f^{-1}(x)$ and $w \in df_y^{-1}(v)$ determines a trivialization $v^\ast \cH \xrightarrow{\sim} k[\eps]^n$, which determines a splitting $s: v^\ast \cH \to H$. Inductively construct two (non-unique) splittings $\phi_P, \phi_Q: H \to v^\ast \cH$ of the quotient map, such that $\phi_P(P^a) \subseteq v^\ast \cP^a$ and $\phi_Q(Q^b) \subseteq v^\ast \cQ^b$ for all $a\in A, b \in B$. These two maps precisely give the differential of $p$:

\begin{equation}
\label{eq:dpy}
dp_y(w) = ( s \circ \phi_P + \Fix P^\sbu,\ s \circ \phi_Q + \Fix Q^\sbu) \in \End H / \Fix P^\sbu \oplus \End H / \Fix Q^\sbu.
\end{equation} 
Now, applying the quotient $q_y$ to this formula, we obtain
\begin{lemma}
\label{lem:deltaxDesc}
For any $x \in S$ and $v \in T_x S$, let $s: v^\ast \cH \to H$ be any $k$-linear splitting, and $\phi_P, \phi_Q: H \to v^\ast \cH$ be any $k$-linear splittings such that $\phi_P(P^a) \subseteq v^\ast \cP^a$ and $\phi_Q(Q^b) \subseteq v^\ast \cQ^b$ for all $a \in A, b \in B$. Then
$\delta_x(v) = [ s \circ \phi_P, s \circ \phi_Q ],$
where the brackets indicate the coset in $M_x$.
\end{lemma}

\subsection{A description of \texorpdfstring{$M_x^\vee$}{the dual space}}
The dual vector space $M^\vee_x$ has a convenient description that will be well-suited to relating it to the coupled Petri map. 

\begin{defn}
\label{defn:rab}
For $a \in A, b \in B$, define $r_{a,b}: M_x \to \Hom(P^a \cap Q^b, H / (P^a + Q^b))$ by 
$$r_{a,b}([\psi_P, \psi_Q]) = q \circ (\psi_P - \psi_Q) \mid_{P^a \cap Q^b},$$
where $q$ denotes the quotient map $H \to H / (P^a + Q^b)$. 
\end{defn} 

\begin{prop}
\label{prop:mdual}
There is an isomorphism
$$\zeta: \sum_{a \in A, b \in B} (P^a + Q^b)^\perp \otimes (P^a \cap Q^b) \to M_x^\vee$$
such that for all $a,b$ the induced map $(P^a + Q^b)^\perp \otimes (P^a \cap Q^b) \to M_x^\vee$ is equal to ${}^t r_{a,b}$. Here, $\perp$ is the annihilator subspace, and $(P^a + Q^b)^\perp \otimes (P^a \cap Q^b)$ is identified with $\Hom(P^a \cap Q^b, H / (P^a + Q^b))^\vee$. 
\end{prop}

\begin{proof}
First, we claim that the product map 
$$r = \prod_{a,b} r_{a,b}: M_x \to \prod_{a \in A, b \in B} \Hom(P^a \cap Q^b, H / (P^a + Q^b) )$$
 is injective. Choose a basis $\mathcal{B}$ adapted to $P^\sbu, Q^\sbu$ (see Fact \ref{fact:permBasis}). Let $[\psi, 0] \in M_x$ be any element of $\ker r$. We will show that $[\psi, 0] = 0$, by constructing $\psi_1 \in \Fix P^\sbu, \psi_2 \in \Fix Q^\sbu$ such that $\psi = \psi_1 - \psi_2$. For each $v \in \mathcal{B}$, choose $a \in A, b \in B$ maximal such that $v \in P^a \cap Q^b$. Then $v = v_1 - v_2$ for some $v_1 \in P^a, v_2 \in Q^b$, since $[\psi, 0] \in \ker r_{a,b}$. Define $\psi_1(v) = v_1,\ \psi_2(v) = v_2$. Then $[\psi,0] = [\psi_1, \psi_2] = 0 \in M_x$. This establishes that $\prod_{a,b} r_{a,b}$ is injective.
 
Now, define $\Delta: \End H \to \prod_{a,b} \Hom(P^a \cap Q^b, H / (P^a + Q^b))$ to be the diagonal map, and $\pi: \End H \twoheadrightarrow M_x$ to be the map $\psi \mapsto [\psi, 0]$. Then $r \circ \pi = \Delta$. Since $r$ is injective and $\pi$ is surjective, we obtain an isomorphism $\zeta: \operatorname{im} {}^t \Delta \xrightarrow{\sim} M^\vee$. Identifying $\operatorname{im} {}^t \Delta$ with $\sum_{a,b} (P^a + Q^b)^\perp \otimes (P^a \cap Q^b)$ gives the desired isomorphism. It follows from the construction that $\zeta_{a,b} = {}^t r_{a,b}$.
\end{proof}

The isomorphism $\zeta$ demonstrates that the map $\delta_x$ may be conveniently studied via the compositions $r_{a,b} \circ \delta_x$, which record first-order deformations of the pair $P^a, Q^b$ of subspaces. These maps have a convenient cohomological description, which is shown in Figure \ref{fig:snake}. Consider the difference morphism $\cP^a \oplus \cQ^b \to \cH$ given by $(\alpha, \beta) \mapsto \alpha - \beta$. Given a tangent vector $v \in T_x S$, we obtain a map of short exact sequences shown in the Figure.
The snake map $P^a \cap Q^b \to H / (P^a + Q^b)$ can be described explicitly as $q \circ s \circ (\phi_P - \phi_Q) \circ \iota$, which is also equal to $r_{a,b} \circ \delta_x$. This proves

\begin{lemma}
\label{lem:rabsnake}
Let $v \in T_x S$ be a tangent vector. For all $a \in A, b \in B$, the map $r_{a,b} \circ \delta_x(v)$ is equal to the snake map described above and shown in Figure \ref{fig:snake}.
\end{lemma}

\begin{figure}
\centering
\begin{tikzcd}
&&&0 \ar[d] & \\
&& \ar[ddd, phantom, ""{coordinate,name=middle}] &P^a \cap Q^b \ar[d, "\iota"] 
\ar[dddll, rounded corners, "r_{a,b}\circ \delta_x(v)",
to path = { -- ([xshift=12ex]\tikztostart.east) \tikztonodes
|- (middle) 
-| ([xshift=-12ex]\tikztotarget.west) 
-- (\tikztotarget)  }
] 
\\
0 \ar[r] & P^a \oplus Q^b \ar[r] \ar[d,crossing over] & v^\ast \cP^a \oplus v^\ast \cQ^b  \ar[r] \ar[d,crossing over] & P^a \oplus Q^b \ar[r] \ar[d, crossing over] \ar[l, bend right, "\phi_P \oplus \phi_Q"'] & 0 \\
0 \ar[r] & H \ar[r] \ar[d, "q"] & v^\ast \cH \ar[r] \ar[l, bend right=15, "s"'] & H \ar[r] & 0\\
& H / (P^a + Q^b) \ar[d] & ~\\
& 0
\end{tikzcd}
\caption{The composition $r_{a,b} \circ \delta_x(v)$ as a snake map.}
\label{fig:snake}
\end{figure}

\subsection{First-order deformation of Brill-Noether flags}

We now specialize to our intended application: let $(C,p,q)$ be a twice-marked smooth curve over an algebraically closed field $k$, $S = \Pic^d(C)$ for some integer $d \geq 2g-1$, and consider the degree-$d$ Brill-Noether flags $\cpb_d, \cqb_d$ in the vector bundle $\cH_d$. For convenience, let $n = d+1-g = \operatorname{rank}(\cH_d$).

Fix a point $[L] \in \Pic^d(C)$. In the notation this section, $H = H^0(C,L)$, $P^a = H^0(C,L(-ap))$, $Q^b = H^0(C,L(-bq))$, $P^a \cap Q^b = H^0(C,L(-ap-bq))$, and the sets of coranks are $A = B = [0, d-2g+1]$. For all $a \in A, b \in B$, there is an exact sequence of $\cO_C$-modules
\begin{equation}
\label{eq:labSequence}
0 \to L(-ap-bq) \to L(-ap) \oplus L(-bq) \to L \to 0,
\end{equation}
where the second map is given by $(s,t) \mapsto s-t$ on sections. Taking cohomology, and using the fact that $L, L(-ap)$, and $L(-bq)$ are nonspecial line bundles, gives an isomorphism
$$ h_{a,b}: H / (P^a + Q^b) \xrightarrow{\sim} H^1(C, L(-ap-bq)).$$

Dualizing these maps and using functoriality of the long exact sequence gives an isomorphism
\begin{equation}
\label{eq:quoth1}
\theta:
\sum_{0 \leq a,b \leq n} H^1(C, L(-ap-bq))^\vee \otimes H^0(L(-ap-bq))
\xrightarrow{\sim}
\sum_{0 \leq a,b \leq n} (P^a + Q^b)^\perp \otimes (P^a \cap Q^b).
\end{equation}
Using Serre duality and identifying the domain of $\theta$ with $T^L_{p,q} |_{[0,n]^2}$ gives an isomorphism
$$\zeta \circ \theta:\ T^L_{p,q} |_{[0,n]^2} \xrightarrow{\sim} M_{[L]}^\vee.$$
Identifying the contangent space $T^\vee_x S \cong H^1(C,\cO_C)^\vee$ with $H^0(C, \omega_C)$, regard ${}^t \delta_{[L]}$ as a map
$${}^t \delta_{[L]}: M^\vee_{[L]} \to H^0(C, \omega_C).$$

We may finally link versality of Brill-Noether flags to the coupled Petri map.

\begin{prop}
\label{prop:etaIsMult}
The map $\eta = {}^t \delta_{[L]} \circ \zeta \circ \theta$ is equal to the $[0,n]^2$-coupled Petri map $\mu^L_{p,q} |_{[0,n]^2}$. 
\end{prop}

\begin{proof}
Let $\eta_{a,b} = \eta \circ \iota$ denote the restriction of $\eta$ to $H^0(C,L(-ap-bq)) \otimes H^0(C,\omega_C \otimes L^\vee(ap + bq) )$. Then ${}^t \eta_{a,b} = {}^t \iota \circ {}^t \theta \circ {}^t \zeta \circ \delta_{[L]} = h_{a,b} \circ r_{a,b} \circ \delta_{[L]} = h_{a,b} \circ s$, where $s$ is examined in Lemma \ref{lem:rabsnake}: it is the map
$$s: T_{[L]} \Pic^d(C) \to \Hom( H^0(C,L(-ap-bq)), H^1(C,L(-ap-bq))$$
such that for all $v \in T_{[L]} \Pic^d(C)$, the map $s(v)$ is equal to the snake map in Lemma \ref{lem:rabsnake}. That snake diagram has a cohomological interpretation. Fix $v \in T_{[L]} \Pic^d(C)$, corresponding to an extension $0 \to L \to \cL \to L \to 0$ of $L$ to a line bundle on $C \times_k \Spec k[\eps]$, and consider the following map of short exact sequences.
\begin{center}
\begin{tikzcd}
0 \ar[r] & L(-ap) \oplus L(-bq) \ar[r] \ar[d] &  \cL(-ap) \oplus \cL(-bq) \ar[r] \ar[d] & L(-ap) \oplus L(-bq) \ar[r] \ar[d] & 0 \\
0 \ar[r] & L \ar[r] & \cL \ar[r] & L \ar[r] & 0\\
\end{tikzcd}
\end{center}
None of these sheaves has higher cohomology, so this diagram gives an acyclic resolution of
\begin{equation}
\label{eq:LabSequence}
0 \to L(-ap-bq) \to \cL(-ap-bq) \to L(-ap-bq) \to 0.
\end{equation}
Taking global sections and forming the snake diagram gives the long exact cohomology sequence of (\ref{eq:LabSequence}), and the snake map of Figure \ref{fig:snake} is equal, up to $h_{a,b}$, to the boundary map 
\begin{equation}
\label{eq:bdyMap}
\partial:\ H^0(C, L(-ap-bq)) \to H^1(C, L(-ap-bq)).
\end{equation}

In other words, $\partial = h_{a,b} \circ r_{a,b} \circ \delta_{[L]}(v)$. 
On the other hand, $\partial$ is also given by taking the cup product with $v$, where we now regard $v$ as an element of $H^1(C, \cO_C)$.
In other words, the map
$${}^t \eta_{a,b}: T_{[L]} \Pic^d(C) \cong H^1(C,\cO_C) \to \Hom( H^0(C,L(-ap-bq)), H^1(C,L(-ap-bq)))$$
is given by
${}^t \eta_{a,b}(v)(\sigma) = v \cup \sigma.$
Dualizing and unwinding definitions, it follows that
$\eta_{a,b}$
is the cup product map. Considering all the $\eta_{a,b}$ together shows that $\eta = \mu^L_{p,q} |_{[0,n]^2}$. 
\end{proof}

\begin{cor}
\label{cor:mudelta}
Let $(C,p,q)$ be a twice-marked smooth curve over an algebraically closed field $k$. The degree-$d$ Brill-Noether flags of $(C,p,q)$ are versal at $[L] \in \Pic^d(C)$ if and only if the $[0,d-2g+1]^2$-coupled Petri map $\mu^L_{p,q} |_{[0,d-2g+1]^2}$ is injective.
\end{cor}

\section{Proof of the versality theorem}
\label{sec:proofs}

We now have all the pieces in place to prove Theorem \ref{thm:coupledPetri} on the $S$-coupled Petri condition, and the versality Theorem \ref{thm:versality}. All that remains is to use some properties of the moduli space of curves to obtain the desired results for all algebraically closed fields.

\begin{proof}[Proof of Theorem \ref{thm:versality}]
If $\pi: \cC \to B$ is a proper flat family of smooth curves with two disjoint sections $P,Q$, the degree-$d$ Brill-Noether flags globalize to $\cpb_d, \cqb_d$ in a vector bundle $\cH_d$ on $\Pic^d(\pi) \to B$. Define a morphism $\Fr(\cH_d) \to \Fl(A; n) \times \Fl(B;n)$ as in Definition \ref{defn:versal}. The locus where this map fails to be smooth is a closed subscheme, as is its image in $B$. So the locus of $b \in B$ where the degree-$d$ Brill-Noether flags are versal is open. Since $\cM_{g,2}$ is irreducible, we need only verify the existence of a twice-marked curve of genus $g$ with versal degree-$d$ Brill-Noether flags over some field in every characteristic. This can be done as follows: find a twice-marked elliptic curve $(C,p,q)$ with $p-q$ non-torsion, chain $g$ such curves together, and deform to an arithmetic surface over a discrete valuation ring with fraction field of the desired characteristic. By Theorem \ref{thm:petriSmoothing}, the geometric general fiber satisfies the fully coupled Petri condition, so by Corollary \ref{cor:mudelta} its Brill-Noether flags (in every degree) are versal.
 \end{proof}

\begin{proof}[Proof of Theorem \ref{thm:coupledPetri}]
Fix a genus $g$, an algebraically closed field $k$, and a finite set $S \subseteq \ZZ \times \ZZ$. Let $\cV \subseteq \cM_{g,2}$ denote the locus of twice-marked curves satisfying the $S$-coupled Petri condition, and for all $d \in \ZZ$ let $\cV_d$ denote the locus of twice-marked curves for which every degree-$d$ line bundle has injective $S$-coupled Petri map. So $\cV = \bigcap_{d \in \ZZ}\ \cV_d$. For all but finitely many $d \in \ZZ$, all elements of $\{d + a + b: (a,b) \in S \}$ are either less than $0$ or greater than $2g-2$; for line bundles $L$ of these degrees the space $T^L_{p,q}(S)$ is trivial, so the $S$-coupled Petri map is injective. Therefore all but finitely many $\cV_d$ are equal to all of $\cM_{g,2}$, so it suffices to fix a single integer $d$ and verify that $\cV_d$ is Zariski-dense.

For all $N \in \ZZ$, the $S$-coupled Petri map of a degree-$d$ line bundle $L$ is equal to the $S'$-coupled Petri map of the degree-$(d+2N)$ line bundle $L(Np+Nq)$, where $S' = \{(a+N,b+N): (a,b) \in S \}$. For $N$ sufficiently large, $S' \subseteq [0,d+2N -2g+1]^2$. By Corollary \ref{cor:mudelta}, $\cV_d$ contains the locus of twice-marked curves with versal degree-$(d+2N)$ Brill-Noether flags. This is dense by Theorem \ref{thm:versality}, hence $\cV_d$ is dense for all $d \in \ZZ$.
\end{proof}

\section{Brill-Noether degeneracy loci}
\label{sec:bnDegen}

This section analyzes the geometry of Brill-Noether degeneracy loci $W^\Pi_d(C,p,q)$ using the versality of Brill-Noether flags, and proves Theorem \ref{thm:bnDegen}. Assume we are in the following situation.

\begin{sit}
\label{sit:bnDegen}
Let $g,d$ be positive integers, $(C,p,q)$ a twice-marked smooth curve of genus $g$, and $\Pi$ a dot array of size $r+1$ with $g-d+r \geq 0$. Also assume that $\Pi \subseteq [0,d]^2$.
\end{sit}

The simplifying assumption $\Pi \subseteq [0,d]^2$ is harmless; see the proof of Theorem \ref{thm:bnDegen}.

We begin by specifying the scheme structure on $W^\Pi_d(C,p,q)$ and interpreting it as a degeneracy locus of Brill-Noether flags. First, we recall the construction of the scheme structure of $W^r_d(C)$ (see e.g. \cite[$\S \mathrm{IV}.3$]{acgh}): the set-theoretic locus $\{ [L] \in \Pic^d(C): h^0(C,L) \geq r+1 \} = \{ [L] \in \Pic^d(C): h^1(C,L) \geq g-d+r \}$ has a natural scheme structure given by the $(g-d+r)$th Fitting ideal of $R^1 \nu_\ast \cL$, where $\nu: C \times \Pic^d(C) \to \Pic^d(C)$ is the projection and $\cL$ is a Poincar\'e line bundle. The Fitting ideal may be computed as a determinantal locus using any resolution by vector bundles. In particular, if we choose any two integers $a,b \geq 2g-1-d$ and let $d' = d+a+b$, we have a resolution $\cP^a_{d'} \oplus \cQ^b_{d'} \to \cH_{d'} \to R^1 \nu_\ast \cL \to 0$, and the Fitting ideal is therefore equal to the determinantal ideal defining the standard scheme structure of
$\{ x \in \Pic^{d'}(C): \dim (\cP^a_{d'})_x \cap (\cQ^b_{d'})_x \geq r+1\},$
as described in Section \ref{sec:flags}.

Now, we define the scheme structure on $W^\Pi_d(C,p,q)$ by a scheme-theoretic intersection
\begin{equation}
\label{eq:wPiScheme}
W^\Pi_d(C,p,q) = \bigcap_{(a,b) \in \NN^2} \tw_{ap+bq} \left( W^{r(a,b)-1}_{d-a-b}(C) \right).
\end{equation}
This in turn may be rephrased as a degeneracy locus of Brill-Noether flags of larger degree. For any two integers $M,N \geq 2g-1$, we may define $d' = d + M + N$ and write equivalently
\begin{equation}
\label{eq:wPiFlags}
\begin{split}
W^\Pi_d(C,p,q) = \tw_{-Mp-Nq} \big(
\{ x \in \Pic^{d'}(C): & \dim (\cP^{M+a}_{d'})_x \cap (\cQ^{N+b}_{d'})_x \geq r^\Pi(a,b)\\& \mbox{ for all } a,b \in [0,d] \} \big).
\end{split}
\end{equation}
The bounds $M,N \geq 2g-1$ ensure that $M+a, N+b \leq d'-2g+1$ for all $a,b \in [0,d]$, so that the flag elements $\cP^{M+a}_{d'}, \cQ^{N+b}_{d'}$ exist.

Now, we wish to write this locus as a degeneracy locus $D_\sigma(\cpb_{d'}; \cqb_{d'})$ of the form discussed in Section \ref{sec:flags}. To do so, we must convert the dot pattern $\Pi$ to a permutation. This requires a few combinatorial preliminaries.

\subsection{\texorpdfstring{$(d,g)$}{(d,g)}-confined permutations}
\label{ss:dgConfined}
\begin{figure}
\centering
\begin{tabular}{cccc}
\begin{tikzpicture}[scale=0.25]
\foreach \x in {-5,...,6} \draw (\x-0.5, 6.5) -- (\x - 0.5, -6.5);
\foreach \y in {-5,...,6} \draw (-6.5, 0.5-\y) -- (6.5, 0.5-\y);
\draw[ultra thick] (-0.5, 6.5) -- (-0.5, -6.5); 
\draw[ultra thick] (-6.5, 0.5) -- (6.5, 0.5);
\fdot{0}{1} \fdot{2}{0} \fdot{3}{3}
\opendot{1}{-3} \opendot{4}{-4} \opendot{5}{-5} \opendot{6}{-6}
\opendot{-1}{-2} \opendot{-2}{-1} 
\opendot{-3}{2} \opendot{-4}{4} \opendot{-5}{5} \opendot{-6}{6}
\draw[ultra thick, dashed] (3.5, -3.5) rectangle (-3.5, 3.5);
\end{tikzpicture}
&
\begin{tikzpicture}[scale=0.25]
\foreach \x in {-5,...,6} \draw (\x-0.5, 6.5) -- (\x - 0.5, -6.5);
\foreach \y in {-5,...,6} \draw (-6.5, 0.5-\y) -- (6.5, 0.5-\y);
\draw[ultra thick] (-0.5, 6.5) -- (-0.5, -6.5); 
\draw[ultra thick] (-6.5, 0.5) -- (6.5, 0.5);
\fdot{0}{1} \fdot{2}{0} \fdot{3}{3}
\opendot{1}{-2} \opendot{4}{-3} \opendot{5}{-4} \opendot{6}{-5}
\opendot{-1}{-1}
\opendot{-2}{2} \opendot{-3}{4} \opendot{-4}{5} \opendot{-5}{6}
\draw[ultra thick, dashed] (3.5, -3.5) rectangle (-2.5, 2.5);
\end{tikzpicture}
&
\begin{tikzpicture}[scale=0.25]
\foreach \x in {-5,...,6} \draw (\x-0.5, 6.5) -- (\x - 0.5, -6.5);
\foreach \y in {-5,...,6} \draw (-6.5, 0.5-\y) -- (6.5, 0.5-\y);
\draw[ultra thick] (-0.5, 6.5) -- (-0.5, -6.5); 
\draw[ultra thick] (-6.5, 0.5) -- (6.5, 0.5);
\fdot{0}{1} \fdot{2}{0} \fdot{3}{3}
\opendot{1}{-1} \opendot{4}{-2} \opendot{5}{-3} \opendot{6}{-4}
\opendot{-1}{2} \opendot{-2}{4} \opendot{-3}{5} \opendot{-4}{6}
\draw[ultra thick, dashed] (3.5, -3.5) rectangle (-1.5, 1.5);
\end{tikzpicture}
& 
\begin{tikzpicture}[scale=0.25]
\foreach \x in {-5,...,6} \draw (\x-0.5, 6.5) -- (\x - 0.5, -6.5);
\foreach \y in {-5,...,6} \draw (-6.5, 0.5-\y) -- (6.5, 0.5-\y);
\draw[ultra thick] (-0.5, 6.5) -- (-0.5, -6.5); 
\draw[ultra thick] (-6.5, 0.5) -- (6.5, 0.5);
\fdot{0}{1} \fdot{2}{0} \fdot{3}{3}
\opendot{1}{2} \opendot{4}{-1} \opendot{5}{-2} \opendot{6}{-3}
\opendot{-1}{4} \opendot{-2}{5} \opendot{-3}{6}
\draw[ultra thick, dashed] (3.5, -3.5) rectangle (-0.5, 0.5);
\end{tikzpicture}
\\
$g-d+r = 2$ & $g-d+r=1$ & $g-d+r=0$ & $g-d+r = -1$\\
\end{tabular}
\caption{Extension of a dot pattern with $r+1 = |\Pi| = 3$ to a $(d,g)$-confined permutation, for various values of $d-g$. The dot pattern in the last example is not $(d,g)$-confined. The dashed squares are explained in Lemma \ref{lem:bijectiveSquare}.}
\label{fig:dpToPerm}
\end{figure}

We wish to extend a dot pattern  $\Pi \subseteq \NN \times \NN$ to a dot pattern in $\ZZ \times \ZZ$ in the ``most efficient way,'' so that a rank function for $L$ may be translated into a rank function for $L(Mp+Nq)$ that results in the same degeneracy locus. 
\begin{defn}
\label{def:dgConfined}
A permutation $\pi: \ZZ \to \ZZ$ is called \emph{$(d,g)$-confined} if both $\pi, \pi^{-1}$ are decreasing on $(-\infty, -1]$ and $\omega_{d-g} \pi$ has finite length.
\end{defn}
Note that the definition of $(d,g)$-confined actually depends only on the difference $d-g$. We may extend Equation \ref{eq:rsigma} 
to $(d,g)$-confined permutations to define, for all $a,b \in \ZZ$,
\begin{equation}
\label{eq:rpi}
r^\pi(a,b) = \# \{ a' \geq a: \pi(a') \geq b\},
\end{equation}
and the requirement that $\omega_{d-g} \pi$ has finite length ensures that $r^\pi(-a,-b) = d+a+b+1-g$ for sufficiently large $a,b$, in accordance with Riemann--Roch.

\begin{lemma}
\label{lem:uniqueConfined}
Fix positive integers $d,g$. For every dot pattern $\Pi$ with $|\Pi| \geq d+1-g$, there is a unique $(d,g)$-confined permutation $\pi$ such that
$\Pi = \{ (a, \pi(a)): a \geq 0 \mbox{ and } \pi(a) \geq 0 \}.$
\end{lemma}

\begin{proof}

We may construct $\pi$ by choosing $N \gg 0$ and extending $\Pi$ to a dot array in the larger square $[-N,d-g+N]^2$ that is the graph of a permutation, and then extending beyond the square by $\pi(n) = d-g-n$ for all other $n$. This is illustrated in Figure \ref{fig:dpToPerm}. The extension of $\Pi$ to a permutation of $[-N,d-g+N]$ is described in \cite[p. 9]{fultonPragacz}. This extension can be done without adding any dots to $[0,N]^2$ provided that there is enough space in the two upper quadrants fill all empty columns, i.e. $N \geq d-g+N+1 - |\Pi|$, or equivalently $|\Pi| \geq d+1-g$. The fourth picture in Figure \ref{fig:dpToPerm} illustrates a case where this inequality fails, and therefore the extension cannot be done without adding a dot to $\Pi$.
\end{proof}

\begin{lemma}
\label{lem:essVersions}
Let $\Pi$ be a dot pattern of size $r+1$, and let $a_0, b_0$ be the minimum row and column occurring in $\Pi$. Let $\pi$ be its $(d,g)$-confined permutation, where $g-d+r \geq 0$. Then $r^\Pi(a,b) = r^\pi(a,b)$ for all $a,b \geq 0$, and $\Ess(\Pi) = \Ess(\pi)$ unless $g-d+r=0$ and $a_0 = b_0 = 0$. In the case $g-d+r=0 , a_0 = b_0 = 0$, we have instead $(0,0) \not\in \Ess(\pi)$ and $\Ess(\Pi) = \Ess(\pi) \cup \{ (0,0) \}$.
\end{lemma}

\begin{lemma}
\label{lem:rhoPerm}
If the dot pattern $\Pi$ has $|\Pi| \geq d+1-g$ and associated $(d,g)$-confined permutation $\pi$, then $\rho_g(d,\Pi) = g - \ell(\omega_{d-g} \pi)$.
\end{lemma}

We omit the straightforward proofs of Lemmas \ref{lem:rhoPerm} and \ref{lem:essVersions}. See Figure \ref{fig:rhoExample} for an illustration of Lemma \ref{lem:rhoPerm}.
\begin{figure}
\centering
\begin{tabular}{cccc}
\begin{tikzpicture}[scale=0.25]
\draw[] (-0.5, 6.5) -- (-0.5, -6.5); 
\draw[] (-6.5, 0.5) -- (6.5, 0.5);
\fdot{0}{1} \fdot{2}{0} \fdot{3}{3}
\opendot{1}{-3} \opendot{4}{-4} \opendot{5}{-5} \opendot{6}{-6}
\opendot{-1}{-2} \opendot{-2}{-1} 
\opendot{-3}{2} \opendot{-4}{4} \opendot{-5}{5} \opendot{-6}{6}
\draw[thick] (1,0) to[in=0, out=90] (-1,2);
\draw[thick] (1,0) to[in=0, out=90] (-2,1);
\draw[thick] (0,-2) to[in=-90, out=180] (-1,2);
\draw[thick] (0,-2) to[in=-90, out=180] (-2,1);
\draw[thick] (3,-3) to[in=-60, out=135] (-1,2);
\draw[thick] (3,-3) to[in=-30, out=135] (-2,1);
\end{tikzpicture}
&
\begin{tikzpicture}[scale=0.25]
\draw[] (-0.5, 6.5) -- (-0.5, -6.5); 
\draw[] (-6.5, 0.5) -- (6.5, 0.5);
\fdot{0}{1} \fdot{2}{0} \fdot{3}{3}
\opendot{1}{-3} \opendot{4}{-4} \opendot{5}{-5} \opendot{6}{-6}
\opendot{-1}{-2} \opendot{-2}{-1} 
\opendot{-3}{2} \opendot{-4}{4} \opendot{-5}{5} \opendot{-6}{6}
\draw[thick] (0,-2) to[in=-90,out=180] (-3,-1);
\draw[thick] (3,-3) to[in=-90,out=180] (-3,-1);
\end{tikzpicture}&
\begin{tikzpicture}[scale=0.25]
\draw[] (-0.5, 6.5) -- (-0.5, -6.5); 
\draw[] (-6.5, 0.5) -- (6.5, 0.5);
\fdot{0}{1} \fdot{2}{0} \fdot{3}{3}
\opendot{1}{-3} \opendot{4}{-4} \opendot{5}{-5} \opendot{6}{-6}
\opendot{-1}{-2} \opendot{-2}{-1} 
\opendot{-3}{2} \opendot{-4}{4} \opendot{-5}{5} \opendot{-6}{6}
\draw[thick] (3,-3) to[in=-60, out=90] (2,3);
\end{tikzpicture}&
\begin{tikzpicture}[scale=0.25]
\draw[] (-0.5, 6.5) -- (-0.5, -6.5); 
\draw[] (-6.5, 0.5) -- (6.5, 0.5);
\fdot{0}{1} \fdot{2}{0} \fdot{3}{3}
\opendot{1}{-3} \opendot{4}{-4} \opendot{5}{-5} \opendot{6}{-6}
\opendot{-1}{-2} \opendot{-2}{-1} 
\opendot{-3}{2} \opendot{-4}{4} \opendot{-5}{5} \opendot{-6}{6}
\draw[thick] (3,-3) to[in=0,out=135] (0,-2);
\draw[thick] (3,-3) to[in=-90,out=135] (1,0);
\end{tikzpicture}
\\
 $(r+1)(g-d+r)$
& $\sum_{i=0}^r (a_i-i) $
& $\sum_{i=0}^r (b_i-i)$ 
& inversions within $\Pi$
\end{tabular}
\caption{The number of non-inversions of $\pi$ is $g-\rho_g(d, \Pi)$.}
\label{fig:rhoExample}
\end{figure}

In the proof of Theorem \ref{thm:bnDegen}, the replacement of a line bundle $L$ by $L(ap + bq)$ has the effect of sliding the graph of $\pi$ down $a$ steps and to the right $b$ steps. We will want to do this in such a way that the portion in $\NN \times \NN$ is the graph of a permutation of $[0,n]$ for some integer $n$. The dashed squares in Figure \ref{fig:dpToPerm} show what we are looking for: squares that enclose the graph of a permutation. The following Lemma formalizes how such squares can be found, and follows from examining the construction in the proof of Lemma \ref{lem:uniqueConfined}.

\begin{lemma}
\label{lem:bijectiveSquare}
Let $\pi$ be the $(d,g)$-confined permutation associated to a dot pattern $\Pi$ of size $r+1$, and let $a_r, b_r$ be the largest row and largest column occurring in $\Pi$, respectively. Then $\pi$ restricts to a bijection $[d-g-b_r, a_r] \to [d-g-a_r, b_r]$, and $\pi(n) = d-g-n$ for all $n$ such that $n < d-g-b_r$ or $n > a_r$. Also, $\pi$ is decreasing on $[a_r, \infty)$, and $\pi^{-1}$ is decreasing on $[b_r, \infty)$. 
\end{lemma}

\begin{cor}
\label{cor:piSlide}
Denote by $\alpha_m$ the ``add $m$'' permutation $\alpha_m(n) = m+n$.
If $M \geq b_r-(d-g)$ and $N \geq a_r - (d-g)$, then the permutation $\alpha_N \pi \alpha_M^{-1}$ restricts to a bijection from $[0, M+N + d-g]$ to itself that is decreasing on $[a_r+M,\infty)$, and its inverse is decreasing on $[b_r+N, \infty)$.
\end{cor}

\subsection{Geometry of Brill-Noether degeneracy loci}
\label{ss:geoDegen}

We now have what we need to study Brill-Noether degeneracy loci as degeneracy loci of Brill-Noether flags. With an eye on Equation \ref{eq:wPiFlags}, we add a bit more notation to Situation \ref{sit:bnDegen}.

\begin{sit}
\label{sit:bnDegen2}
In Situation \ref{sit:bnDegen}, also fix integers $M,N \geq 2g-1$, define $d' = d + M + N$, and define $\pi' = \alpha_N \pi \alpha_M^{-1}$. For convenience, define $n = d' +1 - g$ (the rank of $\cH_{d'}$). 
\end{sit}

If $a_r, b_r$ are the largest row and column occurring in $\Pi$, the assumption $\Pi \subseteq [0,d]$ (from Situation \ref{sit:bnDegen}) ensures $a_r, b_r \leq d$. So Corollary \ref{cor:piSlide} implies that $\pi'$ restricts to a permutation of $[0, n-1]$, and that $\pi'$ and its inverse are both decreasing on $[d' - (2g-1),n-1]$. Both flags $\cpb_{d'}, \cqb_{d'}$ have corank sets $[0, d'-(2g-1)]$ so $\pi'$ meets the monotonicity requirements of Definition \ref{def:dsigma}, and there is a well-defined degeneracy locus $D_{\pi'}(\cpb_{d'}; \cqb_{d'})$. Furthermore, adding degeneracy conditions for pairs $(a,b)$ not in the essential set does not change the scheme structure, so Equation \ref{eq:wPiFlags}, Lemma \ref{lem:essVersions} and the observation that $\Ess(\pi') = \Ess(\pi) + (M,N)$ imply that
\begin{equation}
\label{eq:wPiDegen}
W^\Pi_d(C,p,q) = \tw_{-Mp-Nq} \left( D_{\pi'}(\cpb_{d'}; \cqb_{d;}) \right).
\end{equation}
So Brill-Noether degeneracy loci are just degeneracy loci of Brill-Noether flags, hiding behind a twist. We can now use the versality theorem to prove the local statements from Theorem \ref{thm:bnDegen}, including a stronger form of the smoothness statement.

\begin{thm}
\label{thm:bnDegenLocal}
In Situation \ref{sit:bnDegen2}, assume that $(C,p,q)$ has versal degree-$d'$ Brill-Noether flags. Let $[L]$ be any point of $W^\Pi_d(C,p,q)$. The local dimension of $W^\Pi_d(C,p,q)$ is $\rho_g(d,\Pi)$.

Furthermore, let $L' = L(Mp+Nq)$, and let $\sigma \leq \pi'$ be the permutation associated to the flags $\cpb_{[L']}, \cqb_{[L']}$. Then $W^\Pi_d(C,p,q)$ is singular at $[L]$ if and only if $\sigma$ is singular in $X_{\pi'}$. If $[L] \in \widetilde{W}^\Pi_d(C,p,q)$ then $W^\Pi_d(C,p,q)$ is smooth at $[L]$.
\end{thm}

\begin{proof}
Since smooth morphisms preserve smoothness and local codimension of subschemes, the definition of versality allows us to transfer the study of $[L] \in W^\Pi_d(C,p,q)$ to that of a corresponding point $z \in X_{\sigma}(F^\sbu) \subseteq X_{\pi'}(F^\sbu)$ in a flag variety. The local dimension statement follows since $X_{\pi'}$ has pure codimension $\ell(\omega_{n-1} \pi') = g - \rho_g(d,\Pi)$ in the flag variety.
It remains to show that $\sigma$ is in the smooth locus of $X_{\pi'}$. It is tempting to say that $\pi' = \sigma$, but unfortunately this need not be true if the strata $\cP^a_{d'}, \cQ^b_{d'}$ with $(a+M,b+N) \not\in \EssR(\Pi) \times \EssC(\Pi)$ meet in dimension larger than $r^{\pi'}(a,b)$. One can resolve this by invoking a combinatorial description of the singular locus of $X_{\pi'}$, but we can also accomplish it with a trick: replace $\cpb_{d'}, \cqb_{d'}$ by subflags $\hat{\cP}^\sbu, \hat{\cQ}^\sbu$ in which we only retain the strata of coranks $A = \{a: \pi'(a) > \pi'(a-1) \}$, $B = \{b: \pi'^{-1}(b) > \pi'^{-1}(b-1) \}$. Observe that $A \times B \supseteq \Ess(\pi')$, so the degeneracy locus $D_{\pi'}(\hat{\cP}^\sbu; \hat{\cQ}^\sbu)$ is well-defined and equal to $D_{\pi'}(\cpb_{d'}; \cqb_{d'})$. Also, Lemma \ref{lem:essVersions} implies that $A \subseteq \EssR(\Pi) + M$ and $B \subseteq \EssC(\Pi) + N$, so $[L] \in \widetilde{W}^\Pi_d(C,p,q)$ implies that the permutation associated to the subflags $(\hat{\cP}^\sbu; \hat{\cQ}^\sbu)$ is exactly $\pi'$; thus $[L']$ is a smooth point of $D_{\pi'}(\hat{\cP}^\sbu; \hat{\cQ}^\sbu)$ and $[L]$ is a smooth point of $W^\Pi_d(C,p,q)$.
\end{proof}

\begin{rem}
The proof above follows the techniques of \cite[$\S$4]{cpRR}, although we work here with possibly partial flags. Using \cite[Theorem 4.4]{cpRR} (with mild modifications for partial flags) one can deduce much stronger results about the nature of the singularities of $W^\Pi_d(C,p,q)$; roughly speaking, they are isomorphic, \'etale-locally and up to products with affine space, to the singularities of Schubert varieties. Together with the results of \cite{billey-warrington-maximal, cortez-singularities, kassel-lascoux-reutenauer-singular, manivel-lieu}, this gives essentially a complete description of the singularities of $W^\Pi_d(C,p,q)$.
\end{rem}

\subsection{The Chow class of \texorpdfstring{$W^\Pi_d(C,p,q)$}{W Pi}}
To obtain the existence and intersection-theoretic part of Theorem \ref{thm:bnDegen}, we use the results of \cite{fultonSchubert} about the intersection theory of degeneracy loci (see also the exposition in \cite[$\S$1-2]{fultonPragacz}). We assume throughout this subsection that we are in Situation \ref{sit:bnDegen2}, and furthermore that $W^\Pi_d(C,p,q)$ has the expected dimension $\rho_g(d,\Pi)$ (e.g. it suffices to assume that the degree-$d'$ Brill-Noether flags are versal, by Theorem \ref{thm:bnDegenLocal}).

Translating our notation (indexed by codimension, starting at $0$) into the notation of \cite{fultonSchubert} (indexed by dimension, starting at $1$), our permutation $\pi'$ gives rise to a permutation $w$ of $[1,n]$ via the formula 
$$w(i) = d'+1-g-\pi(i-1) \mbox{ for } 1 \leq i \leq n.$$

Note that $\ell(w) = \ell(\omega_{d'-g} \pi')$ is the codimension of $D_{\pi'}(\cpb_{d'}; \cqb_{d'})$. 

Let $x_1, \cdots, x_{n}$ be Chern roots of the bundle $\cH / \cP^{n}$, ordered such that prefixes of this sequence give Chern roots of $\cH / \cP^a$ for each $a$. Let $y_1, \cdots y_{n}$ be Chern roots of $\cH$ such that prefixes of this sequence give Chern roots of $\cQ^b$ for each $b$. Let $\sch_w$ be the double Schubert polynomial of Lascoux and Sch\"utzenberger. By \cite[Theorem 8.2]{fultonSchubert},
$[D_{\pi'}(\cpb_{d'}; \cqb_{d'})] = \sch_w( x_1, ..., x_n;\ y_1, \cdots, y_n).$
The calculation in \cite[$\S$VII.4]{acgh} shows that all the bundles $\cH_{d'}, \cP^a_{d'}, \cQ^b_{d'}$ have the same total Chern class $e^{-\Theta}$, so the Chern roots simplify considerably: all $x_i = 0$, the elementary symmetric polynomials of the first $g$ $y_i$s are
$e_k(y_1, \cdots, y_g) = (-1)^k \frac{\Theta^k}{k!},$
and the remaining $y_{g+1}, \cdots, y_{n}$ are all $0$.
Since $\sch_w(x,y) = (-1)^{\ell(w)} \sch_{w^{-1}}(y,x)$, and $\sch_w$ has degree $\ell(w)$ (e.g. \cite[$\S$1.3]{fultonPragacz}), 
\begin{equation}
\label{eq:classDSigmaSimp}
\begin{split}
[D_{\pi'}(\cpb_{d'}; \cqb_{d'})] 
&= \sch_{w^{-1}}( -y_1, \cdots, -y_g, 0, \cdots, 0;\ 0, \cdots, 0).
\end{split}
\end{equation}

We may write this class more simply as a (single) Schubert polynomial $\sch_{w^{-1}}(-y_1, \cdots, -y_g)$. 

\begin{lemma}
\label{lem:schubertPoly}
Let $w$ be a permutation of $[1,n]$, and $g$ be an integer such that $w(1) < \cdots < w(g)$ and $g \leq \ell(w)$. Let $e_k$ denote the elementary symmetric polynomial of degree $k$ in $g$ variables, and let $\alpha_1, \cdots, \alpha_g, \Theta$ be elements of a  ring such that
$e_k(\alpha_1, \cdots, \alpha_g) = \frac{\Theta^k}{k!}.$
Then the Schubert polynomial $\sch_w$ satisfies
$\sch_w(\alpha_1, \cdots, \alpha_g, 0, \cdots, 0) = \frac{\Theta^{\ell(w)}}{\ell(w)!} |R(w)|,$
where $R(w)$ denotes the set of reduced words for $w$.
\end{lemma}

\begin{proof}
We use a theorem of Billey, Jockusch, and Stanley. Let $\ell = \ell(w)$, denote by $s_a$ the transposition of $a$ and $a+1$, and identify the elements $a \in R(w)$ by $\ell$-tuples
$a = (a_1, a_2, \cdots, a_{\ell})$
where $w = s_{a_1} s_{a_2} \cdots s_{a_\ell}$. For each $a \in R(w)$, define the set $K(a)$ to be the set of all sequences $i = (i_1, \cdots, i_{\ell})$ of positive integers such that 
\begin{eqnarray}
i_1 \leq i_2 \cdots \leq i_{\ell}, \label{eq:cond1}\\
i_j \leq a_j \mbox{ for } 1 \leq j \leq \ell, \mbox{ and}, \label{eq:cond2}\\
i_j < i_{j+1} \mbox{ if } a_j < a_{j+1}. \label{eq:cond3}
\end{eqnarray}
With these definitions,  \cite[Theorem 1.1]{bjs} states that
$\sch_w = \sum_{a \in R(w)} \sum_{i \in K(a)} x_{i_1} \cdots x_{i_{\ell}}.$
Split $\sch_w$ into a sum $\sch_{w,1} + \sch_{w,2} + \sch_{w,3}$ by partitioning $K(a)$ into three sets as follows. Let $K_1(a)$ consist of those sequences $i \in K(a)$ such that $i_1 < \cdots < i_{\ell} \leq g$. Let $K_2(a)$ consist of those sequences $i \in K(a)$ with $i_{\ell} \leq g$ but with at least one index repeated. Let $K_3(a)$ consist of all $i \in K(a)$ such that $i_{\ell} > g$. The assumption that $w$ is increasing on $1,2,\cdots,g$ implies that $\sch_{w}$ is symmetric in the variables $x_1, \cdots, x_g$ (see e.g. \cite[4.3(iii)]{macdonaldSchubert} or \cite[Corollary 2.11]{fultonSchubert}). The definitions of the three summands $\sch_{w,1}, \sch_{w,2}, \sch_{w,3}$ imply that each one is symmetric in the variables $x_1, \cdots, x_g$ as well; in particular, $\sch_{w,1}$ is an integer multiple of the elementary symmetric polynomial $e_{\ell}$ and $\sch_{w,2}$ is a linear combination of monomial symmetric polynomials with at least one exponent greater than $1$.

Clearly $\sch_{w,3}(\alpha_1, \cdots, \alpha_g, 0, \cdots, 0) = 0$. The expression $\sch_{w,2}(\alpha_1, \cdots, \alpha_g, 0, \cdots, 0)$ may be computed using the \emph{exponential substitution} \cite[\S 7.8]{stanleyv2}. The fact $e_k(\alpha_1, \cdots, \alpha_g) = \frac{\Theta^k}{k!}$ uniquely determines the value of all symmetric polynomials evaluated on $\alpha_1, \cdots, \alpha_g$. By \cite[Proposition 7.8.4(b)]{stanleyv2}, all monomial symmetric polynomials except the $e_k$ evaluate to $0$. Therefore $\sch_{w,2}(\alpha_1, \cdots, \alpha_g, 0, \cdots, 0) = 0$.

Finally, we claim that $\sch_{w,1} = |R(w)| \cdot e_{\ell}$. It suffices to show that for all $a \in R(w)$, $K(a)$ contains all $\ell$-element increasing subsequences of $\{1,2,\cdots,g\}$. Fix a reduced word $a \in R(w)$. For all $j \in \{1, \cdots, \ell\}$, define $w_j = s_{a_1} s_{a_2} \cdots s_{a_j}$. Let $x_j = w_j(a_j)$ and $y_j = w_j(a_j+1)$. Since $a$ is a reduced word, $x_j > y_j$, and $x_j$ occurs before $y_j$ in each of $w_j, w_{j+1}, \cdots, w_{\ell}$. The position $w_k^{-1}(y_j)$ may increase by at most $1$ at a time when $k$ increases, so $w^{-1}(y_j) \leq w_j^{-1}(y_j) + (\ell-j) = a_j +1 - \ell-j$. Since $w$ is assumed to be increasing on $\{1, \cdots, g\}$, it follows that 
\begin{equation}
\label{eq:abound}
a_j \geq g + \ell + j.
\end{equation}
Now, let $i = (i_1, \cdots, i_{\ell})$ be any sequence of positive integers such that $1 \leq i_1 < \cdots < i_{\ell} \leq g$. Then $i$ automatically satisfies conditions (\ref{eq:cond1}) and (\ref{eq:cond2}). Also, for all $j$ the fact that $i$ is strictly increasing implies that $i_j \leq g - \ell + j$. Equation (\ref{eq:abound}) implies that $i_j \leq a_j$, so condition (\ref{eq:cond2}) holds as well. Therefore $K_1(a)$ includes all $\ell$-element increasing subsequences of $\{1,\cdots,g\}$. The reverse inclusion is clear from definitions; it follows that $\sch_{w,1} = |R(w)| \cdot e_{\ell}$ as claimed, and the lemma follows.
\end{proof}

We can now prove the existence and intersection-theoretic part of Theorem \ref{thm:bnDegen}.

\begin{thm}
\label{thm:existence}
In Situation \ref{sit:bnDegen2}, suppose that the degree-$d'$ Brill-Noether flags of $(C,p,q)$ are versal. Then in the Chow ring,
$[ W^\Pi_d(C,p,q) ] = \frac{ | R(\omega_{d-g} \pi) |}{\ell(\omega_{d-g} \pi)!} \Theta^{\ell(\omega_{d-g} \pi)}.$
\end{thm}

\begin{proof}
In the notation above, we have $\ell(\omega_{d-g} \pi) = \ell(w)$. Consider first the case $\ell(w) > g$. Then by Theorem \ref{thm:bnDegenLocal}, $W^\Pi_d(C,p,q)$ is empty, and the formula holds in this case. So assume $\ell(w) \leq g$. 

Corollary \ref{cor:piSlide} implies that $\pi'^{-1}$ is decreasing on $[d'-(2g-1), n-1]$, hence $w^{-1}$ is increasing on $[1,g]$. We have verified the hypotheses of Lemma \ref{lem:schubertPoly}, and it follows via Equation \ref{eq:wPiDegen} that
$[W^\Pi_d(C,p,q)] =  \frac{ \Theta^{\ell(w^{-1})}}{\ell(w^{-1})!} | R(w^{-1}) |.$
Since $w$ and $w^{-1}$ have the same length and there is a bijection between their sets of reduced words, we obtain the desired formula.
\end{proof}

\begin{proof}[Proof of Theorem \ref{thm:bnDegen}]
Suppose that $g,d$ are positive integers, $k$ is an algebraically closed field, and $\Pi$ is a dot array of size at least $d+1-g$. By Theorem \ref{thm:versality}, a general twice-marked smooth curve $(C,p,q)$ has versal degree-$(d+4g-2)$ Brill-Noether flags. If $\Pi \not\subseteq [0,d]^2$, then either $r^\Pi(d+1,0) > 0$ or $r^\Pi(0,d+1) > 0$, so $W^\Pi_d(C,p,q)$ is empty, which accords with the fact that $\rho_g(d,\Pi) < 0$ in this case since either the largest row $a_r$ or largest column $b_r$ in $\Pi$ exceeds $d$. So assume that $\Pi \subseteq [0,d]^2$. Let $M = N = 2g-1$ in Situation \ref{sit:bnDegen2}. If $\rho_g(d,\Pi) < 0$ then $\ell(\omega_{d-g} \pi) > g$ and Theorem \ref{thm:existence} implies that $W^\Pi_d(C,p,q)$ is empty. On the other hand, if $\rho_g(d,\Pi) \geq 0$ then Theorem \ref{thm:existence} shows that $W^\Pi_d(C,p,q)$ is nonempty of the claimed class. Theorem \ref{thm:bnDegenLocal} shows that $W^\Pi_d(C,p,q)$ has pure dimension $\rho_g(d,\Pi)$ and that $\widetilde{W}^\Pi_d(C,p,q)$ is smooth. Since the complement of  $\widetilde{W}^\Pi_d(C,p,q)$ is a union of smaller-dimensional loci $\tw_{-Mp-Nq} \left( D_{\sigma}(\cpb, \cqb) \right)$, $\widetilde{W}^\Pi_d(C,p,q)$ is dense.
\end{proof}

\section*{Acknowledgements}
This work was supported by a Miner D. Crary Sabbatical Fellowship from Amherst College. The work and manuscript were completed during the special semester on Combinatorial Algebraic Geometry at ICERM. I am also grateful to Melody Chan, Montserrat Teixidor i Bigas, David Anderson, and the anonymous referee for helpful conversations and suggestions.

\bibliographystyle{amsalpha}
\bibliography{bnVersalityAG}
\end{document}